\documentclass[reqno, 12pt]{amsart}

\usepackage{amsmath}
\usepackage{amsfonts}
\usepackage{amssymb}
\usepackage{mathtools}
\usepackage{bbm}
\usepackage[bbgreekl]{mathbbol}
\usepackage{graphicx, color}
\usepackage{tikz}
\usetikzlibrary{cd, intersections, calc, decorations.pathmorphing, arrows, decorations.pathreplacing}
\usepackage[all]{xy}
\usepackage[abs]{overpic}

\DeclareSymbolFontAlphabet{\mathbb}{AMSb}
\DeclareSymbolFontAlphabet{\mathbbl}{bbold}

\usepackage{stackengine}
\usepackage{calc}
\newlength\shlength
\newcommand\xshlongvec[2][0]{\setlength\shlength{#1pt}%
  \stackengine{-5pt}{$#2$}{\smash{$\kern\shlength%
    \stackengine{7.1pt}{$\mathchar"017E$}%
      {\rule{\widthof{$#2$}}{.57pt}\kern.4pt}{O}{r}{F}{F}{L}\kern-\shlength$}}%
      {O}{c}{F}{T}{S}}

\usepackage{subfiles}
\usepackage[colorlinks, linkcolor=blue, citecolor=red, urlcolor=cyan]{hyperref}

\numberwithin{equation}{section}

\usepackage{cleveref}
\crefname{thm}{Theorem}{Theorems}
\crefname{cor}{Corollary}{Corollaries}
\crefname{lem}{Lemma}{Lemmas}
\crefname{sublem}{Sublemma}{Sublemmas}
\crefname{prop}{Proposition}{Propositions}
\crefname{dfn}{Definition}{Definitions}
\crefname{defi}{Definition}{Definitions}
\crefname{ex}{Example}{Examples}
\crefname{claim}{Claim}{Claims}
\crefname{conj}{Conjecture}{Conjectures}
\crefname{conv}{Notation}{Notations}
\crefname{rem}{Remark}{Remarks}
\crefname{rmk}{Remark}{Remarks}
\crefname{figure}{Figure}{Figures}
\crefname{section}{Section}{Sections}
\crefname{appendix}{Appendix}{Appendices}

\usepackage{amsthm}
\newtheorem{thm}{Theorem}[section]
\newtheorem{prop}[thm]{Proposition}
\newtheorem{cor}[thm]{Corollary}
\newtheorem{lem}[thm]{Lemma}

\theoremstyle{definition}
\newtheorem{dfn}[thm]{Definition}
\newtheorem{defi}[thm]{Definition}
\newtheorem{ex}[thm]{Example}

\newtheorem{conj}[thm]{Conjecture}
\newtheorem{conv}[thm]{Notation}
\theoremstyle{remark}
\newtheorem{rmk}[thm]{Remark}
\newtheorem{rem}[thm]{Remark}

\newcommand*{\chom}{\mathcal{H}\kern -.5pt om}

\newcommand{\bZ}{\mathbb{Z}}
\newcommand{\bQ}{\mathbb{Q}}
\newcommand{\bR}{\mathbb{R}}
\newcommand{\bC}{\mathbb{C}}

\newcommand{\bS}{\mathbb{S}}
\newcommand{\bT}{\mathbb{T}}

\newcommand{\bN}{\mathbb{N}}

\newcommand{\bP}{\mathbb{P}}
\newcommand{\bG}{\mathbb{G}}

\newcommand{\bE}{\mathbb{E}}

\newcommand{\A}{\mathcal{A}}
\newcommand{\cA}{\mathcal{A}}

\newcommand{\cC}{\mathcal{C}}

\newcommand{\cE}{\mathcal{E}}

\newcommand{\cH}{\mathcal{H}}

\newcommand{\cL}{\Omega}

\newcommand{\X}{\mathcal{X}}
\newcommand{\cX}{\mathcal{X}}

\newcommand{\Z}{\mathcal{Z}}
\newcommand{\cZ}{\mathcal{Z}}

\newcommand{\fS}{\mathfrak{S}}

\newcommand{\Hom}{\mathrm{Hom}}

\newcommand{\sgn}{\mathrm{sgn}}
\newcommand{\trop}{\mathrm{trop}}
\newcommand{\stab}{\mathrm{stab}}

\newcommand{\tr}{\mathsf{T}}
\newcommand{\bep}{\boldsymbol{\epsilon}}

\newcommand{\Teich}{Teichm\"uller}

\DeclareMathOperator{\interior}{\mathrm{int}}
\DeclareMathOperator{\Spec}{\mathrm{Spec}}

\newcommand{\bs}{{\boldsymbol{s}}}
\newcommand{\indedge}{t \overbar{\indk} t'}
\newcommand{\edge}{t \overbar{k} t'}
\newcommand{\indi}{i}
\newcommand{\indj}{j}
\newcommand{\indk}{k}
\newcommand{\numi}{i}

\newcommand{\const}{K''}
\newcommand{\constt}{K'''}
\newcommand{\degr}{\deg}

\makeatletter
\newcommand{\oset}[3][0ex]{%
  \mathrel{\mathop{#3}\limits^{
    \vbox to#1{\kern-2\ex@
    \hbox{$\scriptstyle#2$}\vss}}}}
\makeatother
\newcommand{\overbar}[1]{\oset{#1}{-\!\!\!-\!\!\!-}}

\makeatletter
\newcommand{\osetnear}[3][0ex]{%
  \mathrel{\mathop{#3}\limits^{
    \vbox to#1{\kern-.3\ex@
    \hbox{$\scriptstyle#2$}\vss}}}}
\makeatother

\newcommand\qarrow[2]{\draw[->,shorten >=2pt,shorten <=2pt] (#1) -- (#2) [thick];} 

\tikzset{
  mid arrow/.style={postaction={decorate,decoration={
        markings,
        mark=at position .5 with {\arrow[#1]{stealth}}
      }}},
}

\title{Algebraic entropy of sign-stable mutation loops}

\author[Tsukasa Ishibashi]{Tsukasa Ishibashi}
\address{Tsukasa Ishibashi, Research Institute for Mathematical Sciences, Kyoto University, Kitashirakawa Oiwake-cho, Sakyo-ku, Kyoto 606-8502, Japan.}
\email{ishiba@kurims.kyoto-u.ac.jp}

\author[Shunsuke Kano]{Shunsuke Kano}
\address{Shunsuke Kano, Department of Mathematics, Tokyo Institute of Technology, 2-12-1 Ookayama, Meguro, Tokyo 152-8551, Japan.}
\email{kano.s.ab@m.titech.ac.jp}

\date{\today}

\begin{document}
\maketitle

\begin{abstract}
We introduce a property of mutation loops, called the \emph{sign stability}, with a focus on the asymptotic behavior of the iteration of the tropical $\X$-transformation. A sign-stable mutation loop has a numerical invariant which we call the \emph{cluster stretch factor}, in analogy with that of a pseudo-Anosov mapping class on a marked surface. We compute the algebraic entropies of the cluster $\A$- and $\X$-transformations induced by a sign-stable mutation loop, and conclude that these two coincide with the logarithm of the cluster stretch factor. 
\end{abstract}

\tableofcontents

\section{Introduction}

\subsection{Cluster transformations and their algebraic entropy}
Cluster algebras are at the center of research field initiated by Fomin--Zelevinsky \cite{FZ-CA1} and Fock--Goncharov \cite{FG09}. It has been developed with fruitful connections with other areas of mathematics such as discrete integrable systems \cite{GK13,FM}, \Teich~ theory \cite{FG07,FST08}, and so on. 
The central objects of study are seeds and their mutations.  
A seed consists of two tuples of commutative variables called the $\A$-variables and $\X$-variables, and a matrix called the \emph{exchange matrix}. 
A mutation produces a new seed from a given one, transforming the variables according to the rule determined by the exchange matrix, and changing the exchange matrix to another one at the same time. We call the transformation of $\A$-variables (resp. $\X$-variables) the \emph{cluster $\A$-transformation} (resp. \emph{cluster $\X$-transformation}).
Both are birational transformations. 

A mutation sequence is a finite sequence of seed mutations and permutations of indices. It is called a \emph{mutation loop} if it preserves the exchange matrix. A mutation loop defines an autonomous discrete dynamical system, as the composition of cluster transformations and permutations of coordinates. The mutation loops form a group called the \emph{cluster modular group}, and this group acts on some geometric objects called the \emph{cluster $\A$-} and \emph{$\X$-varieties}, and their tropicalizations by a semifield. Thus the discrete dynamical system induced by a mutation loop takes place in these spaces. 
It is known that many interesting discrete dynamical systems emerge in this way, and in some special cases a geometric construction ensures
integrability
\cite{GK13,FM}. 

As a measure of the deviation from discrete integrability, Bellon and Viallet \cite{BV99} introduced the notion of \emph{algebraic entropy} of a birational map. It is defined as the growth rate of the degree of the reduced rational expression of the iteration of a given map. It is widely believed that Liouville--Arnold integrability corresponds to vanishing of algebraic entropy. Indeed, in \cite{Bel99}, Bellon indicated that the vanishing of entropy should be a necessary condition for the integrability in the Liouville--Arnold sense. For more details, see \cite{HLK19} and references \emph{loc. cit}. 
The algebraic entropy of the cluster $\A$- and $\X$-transformations induced by a mutation loop has been studied by several authors \cite{HI14,FH14,HLK19}.
In \cite{FH14,HLK19}, the authors computed the algebraic entropies of mutation loops of length one, which have been classified by Fordy--Marsh \cite{FM11}. They determined the mutation loops with vanishing entropy among the Fordy--Marsh construction. Moreover explicit first integrals are constructed in each of these cases in \cite{FH14}, showing that the Liouville--Arnold integrability is indeed achieved.

\subsection{Sign stability and the main theorem}
As opposed to the integrable mutation loops discussed above, the \Teich--Thurston theory provides a rich source of \lq\lq non-integrable" mutation loops. The \emph{mapping class group} of an oriented marked surface $\Sigma$ is the group of isotopy classes (\emph{mapping classes}) of orientation-perserving homeomorphisms of $\Sigma$. The Nielsen--Thurston classification theory \cite{Th,FLP} classifies the mapping classes into three types: periodic, reducible (\emph{i.e.}, fixes a simple closed curve on $\Sigma$), and pseudo-Anosov. The last one is of generic type, and interesting for us.
A pseudo-Anosov mapping class is characterized by the existence of a pair of invariant foliations on $\Sigma$, and it is known that the topological entropy of a pseudo-Anosov mapping class is positive. Indeed, these invariant foliations are equipped with a transverse measure which is unique up to positive rescalings. The pseudo-Anosov mapping class rescales these measures by a reciprocal factor  called the \emph{stretch factor}, and the topological entropy is given by the logarithm of the stretch factor. See, for instance, \cite[Section 10.4]{FLP}. In this case, the value of the topological entropy itself is an important numerical invariant of a mapping class.

A deep connection between the \Teich--Thurston theory and the cluster algebra has been known. 
From an ideal triangulation of $\Sigma$ we can form a seed, whose mutation class only depends on the topology of $\Sigma$ \cite{FG07,FST08,Penner}. The mapping class group can be embedded into the corresponding cluster modular group, and the manifold of positive real points of the cluster $\A$- (resp. $\X$-)variety can be identified with the decorated (resp. enhanced) \Teich\  space of $\Sigma$ \cite{FG07,Penner}.
The piecewise-linear manifold of real tropical points of the cluster $\A$- (resp. $\X$-)variety can be identified with the space of decorated (resp. enhanced) measured foliations on $\Sigma$ \cite{FG07,PP93}. 

Based on this correspondence, the first author gave an analogue of the Nielsen--Thurston classification for a general cluster modular group in \cite{Ish19}, which classifies the mutation loops into three types: periodic, cluster-reducible and cluster-pseudo-Anosov. However there exists a slight discrepancy between pseudo-Anosov and cluster-pseudo-Anosov even for a mutation loop given by a mapping class: a pseudo-Anosov mapping class provides a cluster-pseudo-Anosov mutation loop, but the converse is not true.
Therefore, the search for generalized ``pseudo-Anosov" properties of mutation loops continues.

As mentioned above, a pseudo-Anosov mapping class has a pair of invariant foliations. 
A combinatorial model of a measured foliation, called a \emph{train track} is commonly used to study the action of a mapping class on measured foliations. The action can be described by a sequence of \emph{splittings} of the corresponding train tracks. See, for instance, \cite{PH}. 
Our observation is that train track splittings can be translated into tropical cluster transformations.
\begin{figure}[h]
    \centering
    \begin{tikzpicture}[auto, scale=0.8]
    \draw [very thick, red](-4.2,-5.5) .. controls (-3.2,-5.5) and (-3.55,-6.6) .. (-1.7,-6.5) .. controls (0.1,-6.6) and (-0.2,-5.5) .. (0.8,-5.5);
    \draw [very thick, red] (5.1,-5.5) .. controls (6.1,-6) and (8.6,-6) .. (9.6,-5.5);
    \draw [very thick, red](5.1,-7.5) .. controls (6.1,-7) and (8.6,-7) .. (9.6,-7.5);
    \draw [very thick, red](6.1,-7.2) .. controls (7.1,-7) and (7.6,-6) .. (8.6,-5.8);
    \begin{scope}[yscale=-1, shift={(0.75,6.5)}]
    \draw [very thick, red](-4.95,1) .. controls (-3.95,1) and (-4.3,-0.1) .. (-2.45,0) .. controls (-0.65,-0.1) and (-0.95,1) .. (0.05,1);
    \draw [very thick, red]  ;
    \end{scope}
    \node (v1) at (1.8,-6.5) {};
    \node (v2) at (3.8,-6.5) {};
    \draw [->, thick] (v1) edge (v2);
    \node [fill, circle, inner sep=1.5pt, blue] (v4) at (-1.75,-4.4) {};
    \node [fill, circle, inner sep=1.5pt, blue] (v6) at (-1.75,-8.6) {};
    \node [fill, circle, inner sep=1.5pt, blue] (v3) at (-4.5,-6.5) {};
    \node [fill, circle, inner sep=1.5pt, blue] (v5) at (1,-6.5) {};
    \node [fill, circle, inner sep=1.5pt, blue] (v7) at (4.5,-6.5) {};
    \node [fill, circle, inner sep=1.5pt, blue] (v8) at (7.15,-4.4) {};
    \node [fill, circle, inner sep=1.5pt, blue] (v9) at (9.8,-6.5) {};
    \node [fill, circle, inner sep=1.5pt, blue] (v10) at (7.15,-8.6) {};
    \draw [blue](v3) -- (v4) -- (v4) -- (v5) -- (v5) -- (v6) -- (v6) -- (v3);
    \draw [blue](v4) -- (v6);
    \draw [blue](v7) -- (v8) -- (v9) -- (v10) -- (v7) -- (v9);
    \end{tikzpicture}
    \caption{Train track splitting}
\end{figure}
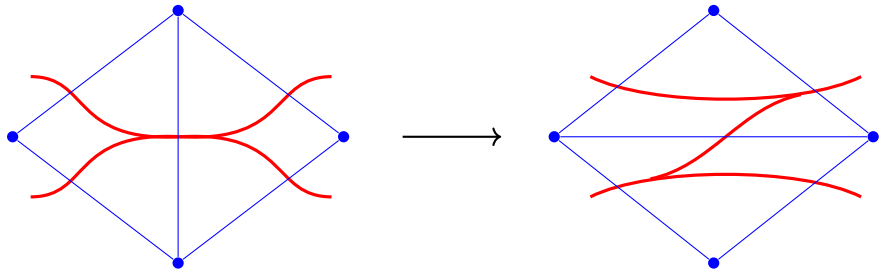
More precisely, some variants of train track splittings and their reverse operations can be unified to \lq\lq signed'' mutations \cite{IN14}, which is obtained by generalizing the usual seed mutation by introducing a sign in the formula. 
Based on these observations we introduce a property of mutation loops called the \emph{sign stability}, which is more closely related to being pseudo-Anosov.
An intuitive (but not exact) definition of the sign stability is a stabilization property of the presentation matrix of the piecewise-linear map obtained as the tropicalization of the cluster $\X$-tranformation.
More precisely, given a mutation sequence and a point of $\X(\bR^\trop)$, we define a sequence of signs indicating which presentation matrices (among three choices at each step of mutation) are applied to that point.
A mutation loop is said to be sign-stable if the sign of each orbit stabilizes to a common one.
To a sign-stable mutation loop, associated is a numerical
invariant which we call the \emph{cluster stretch factor}, which is a positive number greater or equal to $1$. 
Now our main theorem is the following. Let $\cE_\phi^a$ (resp. $\cE_\phi^x$) denote the algebraic entropy (\cref{d:entropy}) of the cluster $\A$- (resp. $\X$-)transformation induced by the mutation loop $\phi$. For a matrix $A$, let $\rho(A)$ denote its spectral radius.

\begin{thm}\label{intro: main theorem}
Let $\phi=[\gamma]_{\bs}$ be a mutation loop with a representation path $\gamma:t_0 \to t$ which is sign-stable (\cref{d:sign stability}) on the set $\cL^{\mathrm{can}}_{(t_0)}$. 
Then we have
\begin{align*}
    \log \rho(\check{E}^{(t_0)}_\phi) &\leq \cE_\phi^a \leq \log R^{(t_0)}_\phi,\\
    \log \rho(E^{(t_0)}_\phi) &\leq \cE_\phi^x \leq \log R^{(t_0)}_\phi.
\end{align*}
Here $R^{(t_0)}_\phi:=\max\{\rho(E^{(t_0)}_\phi),\rho(\check{E}^{(t_0)}_\phi)\}$, where $E^{(t_0)}_\phi$ is the stable presentation matrix (\cref{d:stretch factor}) and $\check{E}^{(t_0)}_\phi:=((E^{(t_0)}_\phi)^\tr)^{-1}$. 
\end{thm}  

\begin{cor}\label{intro: main cor}
Moreover if \cref{p:spec_same} holds true, then we get 
\begin{align*}
    \cE_\phi^a=\cE_\phi^x = \log \lambda^{(t_0)}_\phi.
\end{align*}
Here $\lambda^{(t_0)}_\phi \geq 1$ is the cluster stretch factor (\cref{d:stretch factor}).
\end{cor}
This corollary gives a cluster algebraic analogue of the fact that the topological entropy of a pseudo-Anosov mapping class coincides with the logarithm of the stretch factor. Note that the vanishing of the algebraic entropy corresponds to the equality $\lambda^{(t_0)}_\phi=1$.

Moreover we give several methods for checking the sign stability of a given mutation loop, and demonstrate them in concrete examples. 
See \cref{sec:example}. 
We show that one of them can be effectively applied to certain mutation loops of length one arising from the Fordy--Marsh classification mentioned above. As a byproduct, we obtain a partial confirmation of \cite[Conjecture 3.1]{FH14} for these mutation loops.

\subsection{Related topics and future works}

\paragraph{\textbf{Surface case}}
As mentioned earlier, the sign stability is introduced as a generalization of the pseudo-Anosov property. In fact, it is defined by mimicking the convergence property of the \emph{RLS word} of train track splittings \cite{PP87}. It would be possible to obtain a direct relation in the surface case. 
The tasks are:
\begin{enumerate}
    \item to show that the mutation loop obtained by a pseudo-Anosov mapping class is indeed sign-stable,
    \item to give a direct relation between the sign of a mutation sequence given by a pseudo-Anosov mapping class and the $RLS$ word of the splitting sequence of the corresponding invariant train track.
\end{enumerate}
The task (1) is established in \cite{IK20}. 
As a consequence of (1), the algebraic entropy of the mutation loop obtained by a pseudo-Anosov mapping class would coincide with the topological entropy.


\paragraph{\textbf{Invariance of sign stability and relations with other pseudo-Anosov properties}}
Strictly speaking, the sequence of signs is not an invariant of a mutation loop. Indeed, it highly depends on the choice of mutation sequence which represents a given mutation loop. For example, an elimination or addition of a repeated mutations at the same index does not change the mutation loop but changes (even the size of) the sign sequence. Nevertheless, a large number of experiments indicates that the sign stability is invariant under the change of representation sequence of a mutation loop. Thus we have the following conjecture:

\begin{conj}
Let $\gamma_i: t_i \to t'_i$ $(i=1,2)$ be two edge paths in $\bT_I$ which represent the same mutation loop $\phi:=[\gamma_1]_{\bs} = [\gamma_2]_{\bs}$. Then $\gamma_1$ is sign-stable if and only if $\gamma_2$ is.
\end{conj}
A partial confirmation of this conjecture will be worked out elsewhere.
Moreover we will work on the relations between the sign stability and other properties, such as cluster-pseudo-Anosov property and the asymptotic sign coherence property introduced in \cite{GN19}. 

\paragraph{\textbf{Relation with the categorical entropy}}
A quiver with a non-degenerate potential gives a 3-dimensional Calabi--Yau category as a full subcategory of the derived category of a certain dg algebra.
We can consider it as a categorification of a seed data.
When there exists a non-degenerate potential for a given quiver, a quiver mutation can be lifted to a derived equivalence. However this lifting has an ambiguity on the choice of \emph{signs} of derived equivalences associated to mutations (see \cite{KY}), which corresponds to our signs of mutations.
When a mutation loop is sign-stable, the stable sign determines a canonical lifting. 
It will be interesting to compare the algebraic entropy of a sign-stable mutation loop and the categorical entropy \cite{DHKK} of its canonical lift.

\subsection{Organization of the paper}
In Section 2, basic notions in cluster algebra are recollected, basically following the conventions in \cite{FG09,GHKK}. In Section 3, we introduce the sign stablity of mutation loops and state some basic properties. 
In Section 4, we recall the definition of algebraic entropy following \cite{BV99} and give a proof of \cref{intro: main theorem}. In Section 5, several methods for checking the sign stability for a given mutation loop are proposed. Some concrete examples of sign-stable mutation loops are given, and their cluster stretch factors and algebraic entropies are computed.
\bigskip

\noindent \textbf{Acknowledgement.} 
The authors are grateful to Rei Inoue, Yuma Mizuno and Yoshihiko Mitsumatsu for fruitful discussions. T. I. would like to express his gratitude to his supervisor Nariya Kawazumi for his continuous guidance and encouragement. 
S. K. is also deeply grateful to his supervisor Yuji Terashima for his advice and giving him a lot of knowledge.
T. I. is supported by JSPS KAKENHI Grant Number 18J13304 and the Program for Leading Graduate Schools, MEXT, Japan.

\section{Cluster ensembles}
In this section, we recall basic notions in cluster algebra. Basically we follow the conventions in \cite{FG09,GHKK}. As a technical issue, the distinction between a mutation loop (an element of the \emph{cluster modular group} \cite{FG09}) and its representative path is emphasized.

\subsection{Seed patterns}\label{subsec:seed_mut}
We fix a finite index set $I = \{ 1,2, \dots, N \}$ and a regular tree $\bT_I$ of valency $|I| = N$, whose edges are labeled by $I$ so that the set of edges incident to a fixed vertex has distinct labels.

To each vertex $t$ of $\bT_I$, we assign the following data:
\begin{itemize}
    \item A lattice $N^{(t)} = \bigoplus_{\indi \in I} \bZ e_\indi^{(t)}$ with a basis $(e_\indi^{(t)})_{\indi \in I}$.
    \item An integral skew-symmetric matrix $B^{(t)} = (b_{\indi \indj}^{(t)})_{\indi, \indj \in I}$.
\end{itemize}
We call such a pair $(N^{(t)}, B^{(t)})$ of data a \emph{Fock--Goncharov seed} or simply a \emph{seed}.
Let $M^{(t)}:=\Hom (N^{(t)},\bZ)$ be the dual lattice of $N^{(t)}$, and let $(f^{(t)}_\indi)_{\indi \in I}$ be the dual basis of $(e^{(t)}_\indi)_{\indi \in I}$.
We call the matrix $B^{(t)}$ the \emph{exchange matrix}. 
We define a skew-symmetric bilinear form $\{-,-\}: N^{(t)} \times N^{(t)} \to \bZ$ by $\{e_\indi^{(t)},e_\indj^{(t)}\} := b_{\indi\indj}^{(t)}$.
It induces a linear map $p^*: N^{(t)} \to M^{(t)}$, $n \mapsto \{n,-\}$ called the \emph{ensemble map}.
Each triple $(N^{(t)}, \{-,-\}, (e_\indi^{(t)})_{\indi \in I})$ is called a seed in \cite{FG09}.

\begin{rem}\label{r:matrix convention}
\begin{enumerate}
    \item For simplicity, we only consider seeds with skew-symmetric exchange matrix without frozen indices. See \cite{FG09,GHKK} for a more general definition. 
    \item For exchange matrices we use the notation $B$ rather than $\epsilon$, since we want to reserve the latter for signs $\epsilon \in \{+,0,-\}$. Our exchange matrix is related to the one $B^{\mathrm{FZ}}=(b^\mathrm{FZ}_{\indi\indj})_{\indi.\indj \in I}$ used in \cite{FZ-CA4,NZ12} by the transposition $b^{\mathrm{FZ}}_{\indi\indj} = b_{\indj\indi}$.
\end{enumerate}

\end{rem}

We call such an assignment $\bs: t \mapsto (N^{(t)},B^{(t)})$ a \emph{Fock--Goncharov seed pattern} (or simply a \emph{seed pattern}) if for each edge $\indedge$ of $\bT_I$ labeled by $\indk \in I$, the exchange matrices $B^{(t)}=(b_{\indi\indj})$ and $B^{(t')}=(b'_{\indi\indj})$ are related by the \emph{matrix mutation}:
\[ 
b'_{\indi\indj} = 
    \begin{cases}
    -b_{\indi\indj} & \mbox{if $\indi=\indk$ or $\indj=\indk$}, \\
    b_{\indi\indj} + [b_{\indi\indk}]_+ [b_{\indk\indj}]_+ - [-b_{\indi\indk}]_+ [-b_{\indk\indj}]_+ & \mbox{otherwise}.
    \end{cases} 
\]
Here $[a]_+:=\max\{a,0\}$ for $a \in \bR$, throughout this paper.
As a relation between the lattices assigned to $t$ and $t'$, we consider two linear isomorphisms $(\widetilde{\mu}_\indk^\epsilon)^*: N^{(t')} \xrightarrow{\sim} N^{(t)}$ which depend on a sign $\epsilon \in \{+,-\}$ and is given by
\[
e'_\indi \mapsto 
\begin{cases}
    -e_\indk & \mbox{if $\indi = \indk$},\\
    e_\indi + [\epsilon b_{\indi\indk}]_+ e_\indk & \mbox{if $\indi \neq \indk$}.
\end{cases}
\]
Here we write $e_\indi:=e_\indi^{(t)}$ and $e'_\indi:=e_\indi^{(t')}$. It induces a linear isomorphism $(\widetilde{\mu}_\indk^\epsilon)^*: M^{(t')} \xrightarrow{\sim} M^{(t)}$ (denoted by the same symbol) which sends $f'_\indi$ to the dual basis of $(\widetilde{\mu}_\indk^\epsilon)^*(e'_\indi)$. Explicitly, it is given by
\[
f'_\indi \mapsto
\begin{cases}
    -f_\indk + \sum_{\indj \in I} [-\epsilon b_{\indk\indj}]_+ f_\indj & \mbox{if $\indi = \indk$},\\
    f_\indi & \mbox{if $\indi \neq \indk$}.
\end{cases}
\]
We call each map $(\widetilde{\mu}_\indk^\epsilon)^*$ the \emph{signed seed mutation} at $\indk \in I$.

One can check the following lemma by a direct calculation:
\begin{lem}
The signed mutations are compatible with matrix mutations. Namely, for any $k \in I$ and $\epsilon \in \{+,-\}$, we have
\[
\{(\widetilde{\mu}_\indk^\epsilon)^*(e'_\indi),(\widetilde{\mu}_\indk^\epsilon)^*(e'_\indj)\} = b'_{\indi\indj}.
\]
\end{lem}


\begin{rem}
The map $(\widetilde{\mu}_\indk^+)^*$ corresponds to the seed mutation introduced in \cite{FG09}. However we treat it as a linear isomorphism between two lattices, rather than a base change on a fixed lattice.
\end{rem}

For later discussions, we collect here some properties of the presentation matrices of the signed seed mutation and its dual, with respect to the seed bases. 
For an edge $\indedge$ of $\bT_{I}$ and a sign $\epsilon \in \{+,-\}$, let us consider the matrices $\check{E}^{(t)}_{\indk,\epsilon} = (\check{E}_{\indi\indj})_{\indi,\indj \in I}$ and $E^{(t)}_{\indk,\epsilon} = (E_{\indi\indj})_{\indi,\indj \in I}$, given as follows:
\begin{align*}
    \check{E}_{\indi\indj}:=
    \begin{cases}
        1 & \mbox{if $\indi=\indj \neq \indk$}, \\
        -1 & \mbox{if $\indi=\indj=\indk$}, \\
        [-\epsilon b_{\indk\indj}^{(t)}]_+ & \mbox{if $\indi=\indk$ and $\indj \neq \indk$}, \\
        0 & \mbox{otherwise},
    \end{cases}
\end{align*}

\begin{align*}
    E_{\indi\indj}:=
    \begin{cases}
        1 & \mbox{if $\indi=\indj \neq \indk$}, \\
        -1 & \mbox{if $\indi=\indj=\indk$}, \\
        [\epsilon b_{\indi\indk}^{(t)}]_+ & \mbox{if $\indj = \indk$ and $\indi \neq \indk$}, \\
        0 & \mbox{otherwise}.
    \end{cases}
\end{align*}
Then the transpose of the matrix $E^{(t)}_{\indk,\epsilon}$ gives the presentation matrix of $(\widetilde{\mu}_\indk^\epsilon)^*: N^{(t')} \xrightarrow{\sim} N^{(t)}$ with repsect to the seed bases $(e_\indi^{(t')})$ and $(e_\indi^{(t)})$: $(\widetilde{\mu}_\indk^\epsilon)^*e^{(t')}_\indi= \sum_{\indj \in I}(E^{(t)}_{\indk,\epsilon})_{\indi\indj}e^{(t)}_\indj$.
Similarly the transpose of the matrix $\check{E}^{(t)}_{\indk,\epsilon}$ gives the presentation matrix of $(\widetilde{\mu}_\indk^\epsilon)^*: M^{(t')} \xrightarrow{\sim} M^{(t)}$ with respect to the bases $(f_\indi^{(t')})$ and $(f_\indi^{(t)})$. 

\begin{rmk}\label{rem:elementary matrix}
These matrices look like
\[\check{E}_{\indk,\epsilon}^{(t)} =
\left(\begin{smallmatrix}
1&&&\cdots&&& 0\\
\vdots & \ddots &    &&     && \vdots \\
0&& 1 & 0 &0&& 0 \\
* & \cdots & * & -1 & * & \cdots & *\\
0 & &0& 0 & 1&&0\\
\vdots&&&&& \ddots & \vdots \\
0&&&\cdots&&&1
\end{smallmatrix}\right)
\mbox{\ \ and \ \ }
E_{\indk,\epsilon}^{(t)} =
\left( \begin{smallmatrix}
1&\cdots&0& * &0&\cdots& 0\\
& \ddots && \vdots &&&\\
.&& 1 & * &0&&. \\
\cdot& &0& -1 &0& &\cdot\\
\dot\ & &0& * & 1&&\dot\ \\
&&&\vdots&& \ddots &\\
0&\cdots&0& * &0&\cdots&1
\end{smallmatrix}\right).\]
\end{rmk}

The following are basic properties, which can be checked by a direct computation.
\begin{lem}\label{lem:EF_formulae}
For any edge $\edge$ and $\epsilon \in \{+,-\}$, we have the following equations:
\begin{enumerate}
    \item $(E^{(t)}_{\indk, \epsilon})^{-1} = E^{(t)}_{\indk, \epsilon}$, $(\check{E}^{(t)}_{\indk, \epsilon})^{-1} = \check{E}^{(t)}_{\indk, \epsilon}$.
    \item $(E^{(t)}_{\indk, \epsilon})^{-1} = E^{(t')}_{\indk, -\epsilon}$,
    $(\check{E}^{(t)}_{\indk, \epsilon})^{-1} = \check{E}^{(t')}_{\indk, -\epsilon}$.
    \item $(E^{(t)}_{\indk, \epsilon})^\tr = (\check{E}^{(t)}_{\indk, \epsilon})^{-1}$.
    \item $B^{(t)} \check{E}^{(t)}_{\indk, \epsilon} = {E}^{(t)}_{\indk, \epsilon} B^{(t')}$.
\end{enumerate}
\end{lem}

\begin{conv}\label{conv:check}
In the sequel, we use the notation $\check{A}:=(A^\tr)^{-1}$ for an invertible matrix $A$. \footnote{Note that this is consistent with the notation $\check{E}^{(t)}_{\indk, \epsilon}$, thanks to \cref{lem:EF_formulae} (3). When one considers a skew-symmetrizable exchange matrix, it should be replaced with $\check{A}:=D(A^\tr)^{-1}D^{-1}$ with a positive integral diagonal matrix $D=\mathrm{diag}(d_1,\dots,d_N)$. }
\end{conv}

\subsection{Seed tori}
We are going to associate several geometric objects to a seed pattern $\bs: t \mapsto (N^{(t)},B^{(t)})$. Let $\bG_m:= \Spec \bZ[z,z^{-1}]$ be the multiplicative group. A reader unfamilier with this notation can recognize it as $\bG_m(\Bbbk) = \Bbbk^*$ by substituting a field $\Bbbk$.
We repeatedly use the following:
\begin{lem}
We have an equivalence of categories
\[
\mathsf{Lattices} \xrightarrow{\sim} \mathsf{Tori}; \quad L \mapsto T_L := \Hom(L^\vee,\bG_m).
\]
Here the former is the category of finite rank lattices and the latter is the category of split algebraic tori; $L^\vee:=\Hom (L,\bZ)$ denotes the dual lattice of $L$.
\end{lem}


Indeed, the inverse functor is given by $T \mapsto X_*(T) := \Hom(\bG_m,T)$. We have a natural duality $X_*(T) \cong (X^*(T))^\vee$, where $X^*(T):=\Hom (T,\bG_m)$. 
On the other hand, for a lattice $N$, we have a natural isomorphism $N \xrightarrow{\sim} X^*(T_{N^\vee})$ given by $n \mapsto \mathrm{ch}_n$, where $\mathrm{ch}_n(\phi):=\phi(n)$. 
Taking the dual of both sides and letting $N=L^\vee$, we get 
$L \cong (X^*(T_L))^\vee \cong X_*(T_L)$. 
The character $\mathrm{ch}_n$ on $T_{N^\vee}$ used here is called the \emph{character associated with $n \in N$}.

For each $t \in \bT_I$, we have a pair of tori $\X_{(t)} := T_{M^{(t)}}$, $\A_{(t)} := T_{N^{(t)}}$. They are called the \emph{seed $\X$- and $\A$-tori}, respectively. The characters $X_\indi^{(t)}:=\mathrm{ch}_{e_\indi^{(t)}}$ and $A_\indi^{(t)}:=\mathrm{ch}_{f_\indi^{(t)}}$ are called the \emph{cluster $\X$-} and \emph{$\A$-coordinates}, respectively. 
The ensemble map induces a monomial map $p_{(t)}: \A_{(t)} \to \X_{(t)}$, $p_{(t)}^*(X_\indi^{(t)}) = \prod_{\indj \in I} (A_\indj^{(t)})^{b_{\indi\indj}}$.
\subsection{Cluster transformations and cluster varieties}
Consider an edge $\indedge$ of $\bT_I$. 
\begin{conv}\label{conv:coordinates}
Whenever only one edge $\indedge$ of $\mathbb{T}_I$ is concerned, we denote the cluster coordinates by $X_\indi:=X^{(t)}_\indi$, $A_\indi:=A^{(t)}_\indi$,  $X'_\indi:=X^{(t')}_\indi$ and  $A'_\indi:=A^{(t')}_\indi$.
\end{conv}

Note that the signed mutation induces monomial isomorphisms
$\widetilde{\mu}_\indk^\epsilon: \X_{(t)} \xrightarrow{\sim} \X_{(t')}$ and $\widetilde{\mu}_\indk^\epsilon: \A_{(t)} \xrightarrow{\sim} \A_{(t')}$ are given by
\[
(\widetilde{\mu}_\indk^\epsilon)^*X'_\indi :=
\begin{cases}
    X_\indk^{-1} & \mbox{if $\indi=\indk$}, \\
    X_\indi X_\indk^{[\epsilon b_{\indi\indk}]_+} & \mbox{if $\indi \neq \indk$}, 
\end{cases}
\quad
(\widetilde{\mu}_\indk^\epsilon)^*A'_\indi :=
\begin{cases}
    A_\indk^{-1}\prod_{\indj \in I}A_\indj^{[-\epsilon b_{\indk\indj}]_+} & \mbox{if $\indi=\indk$}, \\
    A_\indi & \mbox{if $\indi \neq \indk$}.
\end{cases}
\]
Pre-composing the birational automorphisms $\mu_\indk^{\#,\epsilon}$
given by
\[
(\mu_\indk^{\#,\epsilon})^*X_\indi := X_\indi(1+X_\indk^\epsilon)^{-b_{\indi\indk}} \quad \mbox{and} \quad (\mu_\indk^{\#,\epsilon})^*A_\indi := A_\indi(1+(p^*X_\indk)^\epsilon)^{-\delta_{\indi\indk}},
\]
we get the \emph{cluster transformations} $\mu_\indk := \widetilde{\mu}_\indk^\epsilon \circ \mu_\indk^{\#,\epsilon}$. Explicitly, they are given by
\[
\mu_\indk^*X'_\indi =
\begin{cases}
    X_\indk^{-1} & \mbox{if $\indi=\indk$}, \\
    X_\indi(1+X_\indk^{-\sgn(b_{\indi\indk})})^{-b_{\indi\indk}} & \mbox{if $\indi \neq \indk$}
\end{cases}
\]
and
\[
\mu_\indk^*A'_\indi =
\begin{cases}
    A_\indk^{-1}(\prod_{\indj \in I} A_\indj^{[b_{\indk\indj}]_+} + \prod_{\indj \in I} A_\indj^{[-b_{\indk\indj}]_+}) & \mbox{if $\indi=\indk$}, \\
    A_\indi & \mbox{if $\indi \neq \indk$},
\end{cases}
\]
which do not depend on the sign $\epsilon$. When we stress the distinction between the $\X$- and $\A$-transformations, we write $\mu_\indk^x$ and $\mu_\indk^a$ instead of $\mu_\indk$.

\begin{rem}
The triple $((B^{(t)})^{\mathrm{FZ}}, (A_\indi^{(t)})_{\indi \in I}, (X_\indi^{(t)})_{\indi \in I})$ forms a seed in the sense of \cite{FZ-CA1}. The variables $A_\indi^{(t)}$ and $X_\indi^{(t)}$ are called \emph{$x$-variable} and \emph{$y$-variable} respectively, in the terminology of Fomin--Zelevinsky.
\end{rem}

\begin{dfn}
The \emph{cluster varieties} $\X_{\bs}$ and $\A_{\bs}$ associated with a seed pattern $\bs: t \mapsto (N^{(t)},B^{(t)})$ is defined by gluing the corresponding tori by cluster transformations:
\[
\X_{\bs} := \bigcup_{t \in \bT_I} \X_{(t)}, \quad \A_{\bs} := \bigcup_{t \in \bT_I} \A_{(t)}.
\]
From the definition, each $\X_{(t)}$ is an open subvariety of $\X_{\bs}$. The pair $(\X_{(t)}, (X_\indi^{(t)})_{\indi \in I})$ of the torus $\X_{(t)}$ and the set of characters $(X_\indi^{(t)})_{\indi \in I}$ is called the \emph{cluster $\X$-chart} associated with $t \in \bT_I$. Similarly we have the notion of \emph{cluster $\A$-charts}.
\end{dfn}

\begin{prop}[{\cite[Proposition 2.2]{FG09}}]
The ensemble maps $p_{(t)}: \A_{(t)} \to \X_{(t)}$ for $t \in \bT_I$ commute with cluster transformations. In particular they induces a morphism $p: \A_{\bs} \to \X_{\bs}$.
\end{prop}
We call the triple $(\A_{\bs},\X_{\bs},p)$ the \emph{cluster ensemble} associated to the seed pattern $\bs$.


\subsection{Horizontal mutation loops}


In this section, we give a definition of a special class of mutation loops.
In brief, general mutation loops are represented by sequences of indices in $I$ and permutations of $I$, but here, we consider the mutation loops which can be represented without permutations. 
We will refer to mutation loops of this type as \emph{horizontal mutation loops}.
It suffices to consider only such mutation loops for the computation of algebraic entropy (see \eqref{eq:entropy_power}). 
We give a concrete definition below.

Fix a seed pattern $\bs: t \mapsto ( N^{(t)},B^{(t)})$.
We say that two vertices $t, t' \in \bT_I$ are \emph{$\bs$-equivalent} (and write $t \sim_\bs t'$) if both vertices are assigned the same matrix: $B^{(t)} = B^{(t')}$.
Then, the following linear isomorphism gives a \emph{seed isomorphism}:
\[
i_{t,t'}^*: (N^{(t')},B^{(t')}) \to (N^{(t)},B^{(t)});\quad e^{(t')}_\numi \mapsto e^{(t)}_\numi.
\]
Namely, it is an isomorphism of lattices with skew-symmetric bilinear forms.
The $\bs$-equivalence class containing $t$ is denoted by $[t]_{\bs}$. 

An edge path $\gamma$ from $t$ to $t'$ in $\bT_I$ is denoted by $\gamma: t \to t'$. 
For such an edge path, we define the birational map $\mu^z_\gamma: \Z_{(t)} \to \Z_{(t')}$ to be the composition of the birational maps associated to the edges it traverses for $(z,\Z)=(a,\A), (x,\X)$.

\begin{rem}\label{r:endpoints}
The map $\mu^z_\gamma$ only depends on the endpoints $t$ and $t'$, thanks to the fact that each cluster transformation is involutive. 
\end{rem}
Let $\gamma_\nu: t_\nu \to t'_\nu$ be a path in $\bT_I$ such that $t_\nu \sim_\bs t'_\nu$ for $\nu=1,2$. 
We say that $\gamma_1$ and $\gamma_2$ are \emph{$\bs$-equivalent} if the following diagram commutes:
\begin{equation}\label{eq:equivalence of paths}
\begin{tikzcd}
    \cZ_{(t'_1)} \ar[r, "i_{t'_1,t_1}^z"] \ar[d, "\mu_{\delta'}^z"'] & \cZ_{(t_1)} \ar[d, "\mu_{\delta}^z"]\\
    \cZ_{(t'_2)} \ar[r, "i_{t'_2,t_2}^z"'] & \cZ_{(t_2)}, 
\end{tikzcd}
\end{equation}
for paths $\delta: t_1 \to t_2$ and $\delta': t'_1 \to t'_2$. 
Here $(z,\cZ) = (a, \cA), (x, \cX)$ and the horizontal maps are induced by the seed isomorphisms $i_{t'_\nu,t_\nu}^*$ for $\nu=1,2$. Note that the commutativity of the diagram does not depend on the choice of paths $\delta$ and $\delta'$. 
The $\bs$-equivalence class containing an edge path $\gamma$ is denoted by $[\gamma]_{\bs}$. 


\begin{dfn}
A \emph{horizontal mutation loop} is the $\bs$-equivalence class of an edge path $\gamma: t \to t'$ such that $t \sim_\bs t'$.
For a horizontal mutation loop $\phi = [\gamma]_\bs$, the path $\gamma$ is called a \emph{representation path} of $\phi$.
\end{dfn}



\paragraph{\textbf{Action on the cluster varieties.}}

For a horizontal mutation loop $\phi$, take a representation path $\gamma:t \to t'$.
Then we have the following composite of birational isomorphisms:
\begin{align}\label{eq:coord_expression}
    \phi^z_{(t)}: \Z_{(t)} \xrightarrow{\mu_\gamma^z} \Z_{(t')} \xrightarrow{i^z_{t',t}} \Z_{(t)}
\end{align}
for $(z,\Z)=(a,\A), (x,\X)$. 
It induces an automorphism on the cluster variety $\Z_\bs$, as follows:
\begin{equation*}
\begin{tikzcd}
    \cZ_{(t)} \ar[r, "\mu_\gamma^z"] \ar[d] & \cZ_{(t')} \ar[d] \ar[r, "i_{t',t}^z"] &\cZ_{(t)} \ar[d]\\
    \cZ_\bs \ar[r, equal] & \cZ_\bs \ar[r, "\phi^z"'] & \cZ_\bs. 
\end{tikzcd}
\end{equation*}
Here the vertical maps are coordinate embeddings given by definition of the cluster variety. If $\gamma_\nu: t_\nu \to t'_\nu$ for $\nu=1,2$ are two representation paths of $\phi$, then the following diagram commutes:
\begin{equation}\label{eq:action compatible}
\begin{tikzcd}
    \cZ_{(t_1)} \ar[r, "\mu_{\gamma_1}^z"] \ar[d, "\mu_\delta^z"] & \cZ_{(t'_1)} \ar[d, "\mu_{\delta'}^z"] \ar[r, "i_{t'_1,t_1}^z"] &\cZ_{(t_1)} \ar[d, "\mu_\delta^z"]\\
     \cZ_{(t_2)} \ar[r, "\mu_{\gamma_2}^z"'] & \cZ_{(t'_2)} \ar[r, "i_{t'_2,t_2}^z"'] &\cZ_{(t_2)},
\end{tikzcd}
\end{equation}
where $\delta: t_1 \to t_2$ and $\delta': t'_1 \to t'_2$ are arbitrary paths. Indeed, the left square commutes by \cref{r:endpoints} and the right square commutes by \eqref{eq:equivalence of paths}. 
Thus the birational actions on $\Z_{\bs}$ induced by different representation paths are compatible with each other, and hence we get a well-defined action of $\phi$ on $\Z_{\bs}$. We call the birational map \eqref{eq:coord_expression} the \emph{coordinate expression} of $\phi$ at the vertex $t_0 \in \bT_I$, which only depends on the mutation loop $\phi$ and the vertex $t_0$.

Later we will use the following notations:
for an edge path $\gamma:t_0 \overbar{k_0} t_1 \overbar{k_1} \cdots \overbar{k_{h-1}} t_h$,
\begin{itemize}
    \item $\mathbf{k}:=(k_0,\dots,k_{h-1})$ and write the path as $\gamma: t_0 \xrightarrow{\mathbf{k}} t_h$.
    \item $h(\gamma) := h$, which is referred to as the  \emph{length}
    of $\gamma$.
    \item $\gamma^n: t_0 \xrightarrow{\mathbf{k}} t_h \xrightarrow{\mathbf{k}} \cdots \xrightarrow{\mathbf{k}} t_{nh}$ for an integer $n \geq 1$. Note that if $\gamma$ represents a horizontal mutation loop $\phi$, then $\gamma^n$ represents a horizontal mutation loop $\psi$ such that $\psi_{(t_0)}^z = (\phi_{(t_0)}^z)^n$. Therefore we write $\phi^n:=\psi$.
\end{itemize}

\begin{rem}\label{r:change of path}
Changes of representation paths of a mutation loop are divided into the following two types:
\begin{itemize}
    \item[(a)] A change with the initial vertex fixed, the path $\delta$ in \eqref{eq:action compatible} being constant. For example, an elimination or addition of a round trip $t \overbar{k} t' \overbar{k} t$ on an edge preserves the $\bs$-equivalence class from \cref{r:endpoints}. Likewise, one can eliminate or add a path $\delta'$ corresponding to one of the \emph{$(h+2)$-gon relations} \cite{FG09}.
    
    \item[(b)] A change of the initial vertex, the paths $\delta:t_1 \xrightarrow{\mathbf{k}} t_2$ and $\delta':t'_1 \xrightarrow{\mathbf{k}'} t'_2$ in \eqref{eq:action compatible} being related as $\mathbf{k}=\mathbf{k}'$. In this case, the birational maps $\phi_{(t_1)}^z$ and $\phi_{(t_2)}^z$ are related by the conjugation of the map $\mu_\delta^z$. 
\end{itemize}
\end{rem}

\begin{rmk}
The origin of the name \lq\lq horizontal" is clarified in our paper \cite{IK20}. A general mutation loop can be formulated as an equivalence class of an edge path on a graph $\bE_I$, which is an enhancement of $\bT_I$ by the Cayley graph of the symmetric group $\fS_I$.
We will call an edge of $\bE_I$ coming from $\bT_I$ (resp. the Cayley graph of $\fS_I$) a \emph{horizontal edge} (resp. \emph{vertical edge}), in analogy with the terminology used for the \emph{mapping class groupoid} \cite[Section 5]{Penner}. 
As we mentioned at the beginning of this subsection, a horizontal mutation loop can be seen as a particular element of the cluster modular group. The notation $\phi^n$ agrees with the composition law in the cluster modular group.
\end{rmk}

\subsection{Separation formulae and the \texorpdfstring{$c$}{c}, \texorpdfstring{$g$}{g}, \texorpdfstring{$f$}{f}-vectors}\label{subsec:vectors}

Fix $t_0 \in \mathbb{T}_I$. Then we assign the \emph{$C$-matrix} $C_t^{\bs;t_0}=(c^{(t)}_{\indi\indj})_{\indi,\indj \in I}$ to each vertex $t \in \mathbb{T}_I$ by the following rule:
\begin{enumerate}
    \item $C_{t_0}^{\bs;t_0}=\mathrm{Id}$,
    \item For each $\indedge$ of $\mathbb{T}_I$, the matrices $C_t^{\bs;t_0}$ and $C_{t'}^{\bs;t_0}$ are related by
    \begin{align}\label{eq:mut_c-mat}
        c'_{\indi\indj}= 
        \begin{cases}
            -c_{\indi\indj} & \mbox{$\indi=\indk$}, \\
            c_{\indi\indj} + [b^{(t)}_{\indi\indk}]_+[c_{\indk\indj}]_+ - [-b^{(t)}_{\indi\indk}]_+[-c_{\indk\indj}]_+ & \mbox{$\indi \neq \indk$}.
        \end{cases}
    \end{align}
\end{enumerate}
Here we write $c_{\indi\indj}:=c^{(t)}_{\indi\indj}$ and $c'_{\indi\indj}:=c^{(t')}_{\indi\indj}$.
Its row vectors $\mathbf{c}^{(t)}_\indi = (c^{(t)}_{\indi\indj})_{\indj \in I}$ are called \emph{$c$-vectors}\footnote{Note that, due to the conventional difference explained in \cref{r:matrix convention}(2), our $C$-matrices are transpose of those used in \cite{FZ-CA4,NZ12}}. 
The following theorem was firstly conjectured in \cite{FZ-CA4}, proved in \cite{DWZ10} for skew-symmetric case, and in \cite{GHKK} for skew-symmetrizable case.

\begin{thm}[Sign-coherence theorem for $c$-vectors]\label{thm:sign_coherence}
For any $t \in \bT_I$ and $\indi \in I$, $\mathbf{c}_\indi^{(t)} \in \bZ_{\geq 0}^I$ or $\mathbf{c}_\indi^{(t)} \in \bZ_{\leq 0}^I$.
\end{thm}
Following \cite{NZ12}, we define the \emph{tropical sign} $\epsilon_{\indi}^{(t)}$ to be $+$ in the former case, and $-$ in the latter case. 

Using the identity $[\epsilon b]_+c + b[-\epsilon c]_+ = [b]_+[c]_+ - [-b]_+[-c]_+$ for real numbers $b,c$ and a sign $\epsilon \in \{+,-\}$, we get 
\begin{align*}
    c'_{\indi\indj} = c_{\indi\indj} + [\epsilon b^{(t)}_{\indi\indk}]_+c_{\indk\indj} + b^{(t)}_{\indi\indk}[-\epsilon c_{\indk\indj}]_+  
\end{align*}
for $\indi \neq \indk$. Substituting $\epsilon=\epsilon^{(t)}_\indk$, we get $c'_{\indi\indj} = c_{\indi\indj} + [\epsilon_\indk^{(t)} b^{(t)}_{\indi\indk}]_+c_{\indk\indj}$. Equivalently, we get the following:
\begin{align}\label{eq:mut_c-mat_2}
C^{\bs;t_0}_{t'} = E_{\indk,\epsilon_\indk^{(t)}}^{(t)} C^{\bs;t_0}_t,
\end{align}
where $E_{\indk,\epsilon_\indk^{(t)}}^{(t)}$ is defined in \cref{subsec:seed_mut}.

Similarly, we assign the \emph{$G$-matrix} $G_t^{\bs;t_0}=(g_{\indi\indj}^{(t)})_{\indi,\indj \in I}$ to each vertex $t \in \mathbb{T}_I$.
The $G$-matrix is originally defined as a grading vector of cluster $\A$-variables. See \cite[Section 6]{FZ-CA4}. Since it involves a bit complicated recurrence relation, we adopt here the simplified recursion given in \cite{NZ12} as the  definition of $G$-matrices:
\begin{enumerate}
    \item $G_{t_0}^{\bs;t_0}$=Id,
    \item For each $\indedge$ of $\mathbb{T}_I$, the matrices $G_t^{\bs;t_0}$ and $G_{t'}^{\bs;t_0}$ are related by
\begin{align}\label{eq:mut_g-mat}
    G^{\bs;t_0}_{t'} = \check{E}_{\indk,\epsilon_\indk^{(t)}}^{(t)} G^{\bs;t_0}_t.
    \end{align}
\end{enumerate}
We refer to the row vectors $\mathbf{g}^{(t)}_\indi$ of $G^{\bs;t_0}_t$ as \emph{$g$-vectors}.
The \emph{tropical duality} 
\begin{align}\label{eq:tropical duality}
    G^{\bs;t_0}_t = \check{C}^{\bs;t_0}_t
\end{align}
is a consequence of \cref{lem:EF_formulae} (3), \eqref{eq:mut_c-mat_2} and \eqref{eq:mut_g-mat}.

In \cite[Section 3]{FZ-CA4}, they introduce the \emph{$F$-polynomials} as the special values of cluster $\A$-variables (called ``$x$-variables'' in \emph{loc. cit.}) with principal coefficients.
In this paper, we adopt the recurrence relation discussed in \cite[Section 5]{FZ-CA4} as the definition of $F$-polynomial.
Fix a vertex $t_0 \in \bT_I$ and $N$ indeterminates $y_1, \dots, y_N$.
We assign the \emph{$i$-th $F$-polynomials} $F^{(t)}_i(y_1, \dots, y_N) \in \bZ[y_1, \dots, y_N]$ for $i \in I$ to each vertex $t \in \bT_I$:
\begin{enumerate}
    \item $F^{(t_0)}_i = 1$ for all $i \in I$,
    \item For each $\edge$ of $\bT_I$, the polynomials $(F^{(t)}_i)_{i \in I}$ and $(F^{(t')}_i)_{i \in I}$ are related by
    \[
    F^{(t')}_i = 
    \begin{cases}
        F^{(t)}_i & \mbox{if } i \neq k,\\
        \displaystyle \frac{\prod_{j \in I} y_j^{[c^{(t)}_{kj}]_+} \prod_{l \in I}(F^{(t)}_l)^{[b^{(t)}_{kl}]_+} + \prod_{j \in I}y_j^{[-c^{(t)}_{kj}]_+} \prod_{l \in I}(F^{(t)}_l)^{[-b^{(t)}_{kl}]_+}}{F^{(t)}_k} & \mbox{if } i = k.
    \end{cases}
    \]
\end{enumerate}
Though it is not clear that $F^{(t)}_i(y_1, \dots, y_N)$ are polynomials from the above definition, one can check it by following the discussions in \cite{FZ-CA4}; it is one of the consequences of the Laurent phenomenon of cluster $\A$-variables \cite[Proposition 3.6]{FZ-CA4}.

Using these concepts, we can separate the rational expression of $\A$- and $\X$-variables in initial variables into its \emph{monomial part} and \emph{polynomial part}. The following formulae are called the \emph{separation formulae}:
\begin{thm}[{\cite[Proposition 3.13, Corollary 6.3]{FZ-CA4}}]\label{t:separation}
Fix a vertex $t_0 \in \bT_I$ and write $A_\indi:=A_\indi^{(t_0)}$ and $X_\indi:=X_\indi^{(t_0)}$ for $\indi \in I$. Then for each $t \in \bT_I$, the variables $A_\indi^{(t)}$ and $X_\indi^{(t)}$ can be written as follows:
\begin{align}
    A_\indi^{(t)} &= \prod_{\indj=1}^N A_\indj^{g_{\indi \indj}^{(t)}} \cdot F_\indi^{(t)} (p^*X_1,\dots,p^*X_N),\label{eq:separation_A} \\
    X_\indi^{(t)} &= \prod_{\indj=1}^N X_\indj^{c_{\indi \indj}^{(t)}}  F_\indj^{(t)}(X_1,\dots,X_N)^{b_{\indj \indi}^{(t)}}.\label{eq:separation_X}
\end{align}
\end{thm}
The following lemma will be used to give an estimate of the algebraic entropy of the $\X$-transformation from below:

\begin{lem}[{\cite[Proposition 5.2]{FZ-CA4}}]\label{l:no cancellation}
Each of the $F$-polynomials $F^{(t)}_i(y_1,\dots,y_N)$ is not divisible by any $y_j$.
\end{lem}

Fujiwara and Gyoda introduce the \emph{$F$-matrices} as a linearization of $F$-polynomials in \cite{FuGy19}.

\begin{dfn}
Fix a vertex $t_0 \in \bT_I$.
For each $i \in I$ and $t \in \bT$, let $f^{(t)}_{i1}, \dots, f^{(t)}_{iN}$ denote the maximal degrees of $y_1, \dots, y_N$ in the $i$-th $F$-polynomial $F^{(t)}_i(y_1, \dots, y_N)$.
Then $\mathbf{f}_i^{(t)}:= (f^{(t)}_{i1}, \dots, f^{(t)}_{iN})$ is called the $f$-vector and $F^{\bs;t_0}_t := (f^{(t)}_{ij})_{i,j \in I}$ is called the \emph{$F$-matrix} assigned at $t$.
\end{dfn}

Later we use some of the mutation formulae for $F$-matrices derived in \cite{FuGy19}. 
For a seed pattern $\bs$, let $-\bs$ denote the seed pattern $t \mapsto (N^{(t)}, -B^{(t)})$.

\begin{thm}[{\cite[Theorem 2.8]{FuGy19}}]\label{t:CGF_of_-bs}
Let $\bs$ be a seed pattern and fix $t_0\in \bT_I$.
Then for each $t \in \bT_I$, we have the following equations:
\begin{align*}
    C^{-\bs;t_0}_t &= C^{\bs;t_0}_t + B^{(t)} F^{\bs:t_0}_t,\\
    G^{-\bs;t_0}_t &= G^{\bs;t_0}_t +  F^{\bs:t_0}_t B^{(t_0)},\\
    F^{-\bs;t_0}_t &= F^{\bs;t_0}_t.
\end{align*}
\end{thm}

\begin{thm}[{\cite[Proposition 2.16]{FuGy19}}]\label{t:FuGy}
For any edge $\indedge$ in $\bT_I$, we have
\begin{align*}
    F^{\bs;t_0}_{t'} = \check{E}_{\indk,\epsilon_\indk^{(t)}}^{(t)} F^{\bs;t_0}_t + [\epsilon_\indk^{(t)}C^{-\bs,t_0}_t]_+^{\indk\bullet}.
\end{align*}
Here for an $N \times N$ matrix $A$, 
\begin{align*}
    [A]^{k \bullet}:=\mathrm{diag}(0,\dots,0,\overset{k}{1},0,\dots,0) \cdot A
\end{align*}
and $[A]_+$ is the matrix obtained by applying $[-]_+$ to each entries.
\end{thm}

\subsection{Tropicalizations of the cluster ensemble}
Let $\mathbb{P}=(\mathbb{P},\oplus, \cdot)$ be a semifield.
For a torus $T_L$ with finite rank lattice $L$, we define $T_L(\bP) := L \otimes_\bZ \bP^\times$. Here $\bP^\times:=(\bP,\cdot)$ denotes the multiplicative group. 
A positive rational map $f:T_L \to T_{L'}$ naturally induces a map $f(\bP): T_L(\bP) \to T_{L'}(\bP)$.
For a more detailed (and generalized) correspondence, see  \cite{GHKK}.

Recall the character $\mathrm{ch}_{\ell^\vee} \in X^*(T_{L})$ associated with a point $\ell^\vee \in L^\vee$. It induces a group homomorphism $\mathrm{ch}_{\ell^\vee}(\mathbb{P}): T_L(\mathbb{P}) \to \mathbb{P}$ by $\psi \otimes p \mapsto \langle \psi,\mathrm{ch}_{\ell^\vee} \rangle p$. One can verify that it coincides with the evaluation map $L \otimes_\bZ \mathbb{P} \to \mathbb{P}$; $\lambda \otimes p \mapsto \ell^\vee(\lambda)p$.
Applying them to seed tori $\X_{(t)}$ and $\A_{(t)}$, we get $\X_{(t)}(\mathbb{P})= M^{(t)} \otimes \mathbb{P}$ and $\A_{(t)} = N^{(t)} \otimes \mathbb{P}$ equipped with functions 
\begin{align*}
    x^{(t)}_\indi:=\mathrm{ch}_{e^{(t)}_\indi}(\mathbb{P}): \X_{(t)}(\mathbb{P}) \to \mathbb{P}, \quad a^{(t)}_\indi:=\mathrm{ch}_{f^{(t)}_\indi}(\mathbb{P}): \A_{(t)}(\mathbb{P}) \to \mathbb{P}
\end{align*}
which we call the \emph{tropical cluster $\X$-} and \emph{$\A$-coordinates}. 
Since cluster transformations are positive rational maps, they induce maps between these sets. 

\begin{dfn}
We define the set of $\mathbb{P}$-valued points as $\X_{\bs}(\mathbb{P}) := \bigsqcup_{t \in \bT_I} \X_{(t)}(\mathbb{P}) \slash \sim$, where for each edge $\indedge$, two points $x \in \X_{(t)}(\mathbb{P})$ and $x' \in \X_{(t')}(\mathbb{P})$ are identified if $x' = \mu_\indk(\mathbb{P})(x)$. Similarly we define $\A_{\bs}(\mathbb{P})$.
\end{dfn}
We are mainly interested in the case $\mathbb{P}=\bZ^\trop$ or $\bR^\trop$. These semifields are defined to be the sets $\bZ$ and $\bR$ equipped with \emph{min-plus} operations $a \oplus b := \min\{a,b\}$, $a \cdot b := a+b$. 
In these cases, the tropicalized cluster transformations associated with an edge $\indedge$ of $\mathbb{T}_I$ are given by

\begin{align}\label{eq:trop x-cluster}
    (\mu_\indk^{x}(\mathbb{P}))^*x'_\indi =
\begin{cases}
    -x_\indk & \mbox{if $\indi=\indk$}, \\
    x_\indi-b_{\indi\indk}\min\{0,-\sgn(b_{\indi\indk})x_\indk\} & \mbox{if $\indi \neq \indk$}
\end{cases}
\end{align}
and
\begin{align}\label{eq:trop a-cluster}
    (\mu_\indk^{a}(\mathbb{P}))^*a'_\indi =
\begin{cases}
    -a_\indk+ \min \left\{\sum_{\indj \in I}[b_{\indk\indj}]_+a_\indj, \sum_{\indj \in I}[-b_{\indk\indj}]_+a_\indj \right\} & \mbox{if $\indi=\indk$}, \\
    a_\indi & \mbox{if $\indi \neq \indk$}.
\end{cases}
\end{align}
Here we abbreviate \cref{conv:coordinates} for tropical coordinates.
In particular, $\cX_\bs(\bR^\trop)$ and $\cA_\bs(\bR^\trop)$ are piecewise linear (PL for short) manifold.
The ensemble map $p_{(t)}: \A_{(t)} \to \X_{(t)}$ induces a linear map $p_{(t)}(\mathbb{P}): \A_{(t)}(\mathbb{P}) \to \X_{(t)}(\mathbb{P})$, which is given by $(p_{(t)}(\mathbb{P}))^* x_\indk = \sum_{\indi \in I} b_{\indk\indi}a_\indi$.
\section{Sign stability of mutation loops}
In this section, we introduce the sign stability of (horizontal) mutation loops and state some basic properties. 
\subsection{Definition of the sign stability}

In what follows, we mainly deal with the $\bR^\trop$ or $\bZ^\trop$-valued points of cluster varieties. Therefore we omit the symbol $\bR^\trop$ or $\bZ^\trop$ from the tropicalizations of positive maps, for notational simplicity. Moreover we omit the symbol \lq\lq $a$'' and \lq\lq $x$'' from the superscript when no confusion can occur.

In order to obtain the presentation matrices of the tropical cluster $\X$-transformation in the case $\mathbb{P}=\bR^\trop$, first we rewrite the formula \eqref{eq:trop x-cluster}.
For a real number $a \in \bR$, let $\sgn(a)$ denote its sign:
\[
\sgn(a):=
\begin{cases}
    + & \mbox{ if } a>0,\\
    0 & \mbox{ if } a=0,\\
    - & \mbox{ if } a<0.
\end{cases}
\]

\begin{lem}\label{l:x-cluster signed}
Fix a point $w \in \X_{(t)}(\bR^\trop)$. 
Then the tropical cluster $\X$-transformation  \eqref{eq:trop x-cluster} can be written as follows:
\begin{align}\label{eq:sign x-cluster}
    x'_\indi(\mu_\indk(w)) =
\begin{cases}
    -x_\indk(w) & \mbox{if $\indi=\indk$}, \\
    x_\indi(w)+[\sgn(x_\indk(w))b_{\indi\indk}]_+x_\indk(w) & \mbox{if $\indi \neq \indk$}.
\end{cases}
\end{align}
\end{lem}

\begin{proof}
Using the formula
$a[\sgn (a)b]_+= b[\sgn (b)a]_+$ for real numbers $a,b$, 
we get 
\begin{align*}
    -b_{\indi\indk}\min\{0,-\sgn(b_{\indi\indk})x_\indk(w)\} = b_{\indi\indk}[\sgn (b_{\indi\indk}) x_\indk(w)]_+ = x_\indk(w)[\sgn (x_\indk(w)) b_{\indi\indk}]_+.
\end{align*}
\end{proof}
With this lemma in mind, we consider the half-spaces
\begin{align*}
    \cH_{\indk,\epsilon}^{x,(t)}:= \{ w \in \X_{(t)}(\bR^\trop) \mid \epsilon x^{(t)}_\indk(w) \geq 0 \}
\end{align*}
for $\indk \in I$, $\epsilon \in \{+,-\}$ and $t \in \mathbb{T}_I$.

Let $\psi:V \to W$ be a PL map between two vector spaces with fixed bases. If $\psi$ is differentiable at $x \in V$, then the \emph{presentation matrix} of $\psi$ at $x$ is the presentation matrix of the tangent map $d\psi_x:T_x V \to T_{\psi(x)}W$ with respect to the given bases. When $V=\X_{(t)}(\bR^\trop)$ and $W=\X_{(t')}(\bR^\trop)$ for some $t,t' \in \bT_I$, we always consider the bases $(f_i^{(t)})_{i \in I}$ and $(f_i^{(t')})_{i \in I}$ respectively, unless otherwise specified.

Then we have the following immediate corollary of \cref{l:x-cluster signed}.
\begin{cor}\label{cor:x_and_E}
For $\epsilon \in \{+,-\}$, the tropical cluster $\X$-transformation $\mu_\indk: \X_{(t)}(\bR^\trop) \to \X_{(t')}(\bR^\trop)$ is differentiable at any point in $\interior\mathcal{H}_{\indk,\epsilon}^{x,(t)}$, and its presentation matrix is given by $E_{\indk,\epsilon}^{(t)}$ there.~ \footnote{Note that this is the presentation matrix of the signed mutation $(\widetilde{\mu}_k^\epsilon)^*: M^{(t')} \to M^{(t)}$. It should be understood as the transpose of the signed mutation $(\widetilde{\mu}_k^\epsilon)^*: N^{(t')} \to N^{(t)}$ in view of \cref{lem:EF_formulae} (1) and (3) with a notice that the lattice $N^{(t)}$ gives functions on $\X_{(t)}(\bR^\trop) \cong M^{(t)} \otimes \bR$.}
\end{cor}

    
    

We are going to define the \emph{sign} of a path in $\bT_I$. In the sequel, we use the following notation.
\begin{conv}\label{path convention}
\begin{enumerate}
    \item For an edge path $\gamma:t_0 \overbar{k_0} t_1 \overbar{k_1} \cdots \overbar{k_{h-1}} t_h$ and 
$i=1,\dots,h$, let $\gamma_{\leq i}: t_0 \xrightarrow{(k_0,\dots,k_{i-1})} t_{i}$ be the sub-path of $\gamma$ from $t_0$ to $t_{i}$, and let $\gamma_{\leq 0}$ be the constant path at $t$.

\item Fixing the initial vertex $t_0 \in \bT_I$ in the sequel, 
we simply denote the coordinate expression of $\phi$ at  $t_0$ by $\phi:=\phi_{(t_0)}: \X_{(t_0)}(\bR^\trop) \to \X_{(t_0)}(\bR^\trop)$. For a point $w \in \X_{(t_0)}(\bR^\trop)$, let $E_\phi^{(t_0)}(w)$ denote the presentation matrix of $\phi$ at $w$.
\end{enumerate}
\end{conv}


\begin{dfn}[sign of a path]\label{d:sign}
Let the notation as above, and fix a point $w \in \X_{(t_0)}(\bR^\trop)$.
Then the \emph{sign} of $\gamma$ at $w$ is the sequence $\boldsymbol{\epsilon}_\gamma(w)=(\epsilon_0,\dots,\epsilon_{h-1}) \in \{+,0,-\}^h$ of signs defined by
\[\epsilon_{i} := \sgn(x^{(t_{i})}_{k_i}(\mu_{\gamma_{\leq {i}}} (w)))\]
for $i=0,\dots,h-1$.
\end{dfn}


Next lemma expresses the heart of \cref{d:sign}. 

\begin{lem}\label{lem:pres_mat_x}
Let $\boldsymbol{\epsilon}=(\epsilon_0,\dots,\epsilon_{h-1})$ be the sign of a path $\gamma$ at $w \in \cX_{(t_0)}(\bR^\trop)$. If it is strict, namely $\boldsymbol{\epsilon} \in \{+,-\}^h$, then the cluster $\X$-transformation $\mu_\gamma$ is differentiable at $w$, and the presentation matrix is given by $E^{\boldsymbol{\epsilon}}_\gamma := E_{k_{h-1},\epsilon_{h-1}}^{(t_{h-1})}\dots E_{k_1,\epsilon_1}^{(t_1)} E_{k_0,\epsilon_0}^{(t_0)}$.
\end{lem}

Using the concept of the sign of a mutation sequence, now we define the sign stability.

\begin{dfn}[sign stability]\label{d:sign stability}
Let $\gamma$ be a path as above which represents a mutation loop $\phi:=[\gamma]_{\bs}$. Let $\cL \subset \X_{(t_0)}(\bR^\trop)$ be a subset which is invariant under the rescaling action of $\bR_{> 0}$. 
Then we say that $\gamma$ is \emph{sign-stable} on $\cL$ if there exists a sequence $\boldsymbol{\epsilon}^\stab_{\gamma,\cL} \in \{+,-\}^h$ of strict signs such that for each $w \in \cL \setminus \{0\}$, there exists an integer $n_0 \in \mathbb{N}$ satisfying   \[\boldsymbol{\epsilon}_\gamma(\phi^n(w)) = \boldsymbol{\epsilon}^\stab_{\gamma, \cL} \]
for all $n \geq n_0$. 
We call $\boldsymbol{\epsilon}_{\gamma,\cL}^\stab$ the \emph{stable sign} of $\gamma$ on $\cL$.
\end{dfn}
For example, if $\phi$ has an invariant ray $\bR_{\geq 0}w$ such that the sign $\boldsymbol{\epsilon}_\gamma(w)$ is strict, then $\gamma$ is sign-stable on $\cL:=\bR_{\geq 0}w$.
More interesting choices of $\cL$ would be the set $\bR_{>0}\cdot \X_{(t_0)}(\bZ^\trop)$ of integral points (cf. \cite{GN19}) or the union $\cL_{(t_0)}^{\mathrm{can}}$ of the positive and negative cones which will be introduced below.
See \cref{sec:example} for concrete examples. 
As a simple non-example, if $\phi$ has an invariant ray $\bR_{\geq 0}w$ such that the sign $\boldsymbol{\epsilon}_\gamma(w)$ is non-strict (\emph{i.e.}, contains $0$), then it cannot be sign-stable on any set $\cL$ which contains $\bR_{\geq 0}w$.

Sign stability in particular implies that the presentation matrix of $\phi$ at each point $w \in \cL$ stabilizes:

\begin{cor}\label{cor:asymptotic linearity}
Suppose $\gamma$ is a path as above which represents a mutation loop $\phi=[\gamma]_{\bs}$, and which is sign-stable on $\cL$.
Then there exists an integral $N \times N$-matrix $E^{(t_0)}_{\phi,\cL} \in GL_N(\bZ)$ such that for each $w \in \cL$, there exists an integer $n_0 \geq 0$ such that $E_\phi^{(t_0)}(\phi^n(w)) = E^{(t_0)}_{\phi,\cL}$ for all $n \geq n_0$.
\end{cor}
We will discuss a Perron--Frobenius property of the stable presentation matrix $E^{(t_0)}_{\phi,\cL}$ in \cref{subsec:PF property}.

Next lemma shows that the tropical sign for $c$-vectors can be regarded as a special case of the sign of a path $\gamma$ in $\bT_I$.

\begin{lem}\label{cor:C_as_pres_mat}
When the coordinates of $w \in \X_{(t_0)}(\bR^\trop)$ are positive, the sign $\boldsymbol{\epsilon}_\gamma(w)$ coincides with the sequence of tropical signs
\[\boldsymbol{\epsilon}^\trop_\gamma := (\epsilon^{(t_0)}_{k_0}, \dots, \epsilon^{(t_{h-1})}_{k_{h-1}} ).\]
Moreover, the PL action of $\phi$ is differentiable at any point in the interior of the non-negative cone
\[
    \cC^+_{(t_0)}:= \left\{ w \in \X_{(t_0)}(\bR^\trop) \mid x_i^{(t_0)}(w) \geq 0 \mbox{ for $i =1,\dots,N$} \right\},
\]
and its presentation matrix there coincides with the $C$-matrix $C^{\boldsymbol{s}; t_0}_{t_h}$:
\[
    \phi|_{\cC^+_{(t_0)}} = E^{\boldsymbol{\epsilon}^\trop_\gamma}_\gamma  = C^{\boldsymbol{s}; t_0}_{t_h}.
\]
Moreover, we have $\check{E}^{\boldsymbol{\epsilon}^\trop_\gamma}_\gamma  = G^{\boldsymbol{s}; t_0}_{t_h}$. 
\end{lem}

\begin{proof}
For $t \in \bT_I$, $l_{(t)}^\pm \in \X_{(t)}(\bR^\trop)$ be the unique point determined by $\boldsymbol{x}_{(t)}(l_{(t)}^\pm)=(\pm 1,\dots,\pm 1)^\tr$. Here $\boldsymbol{x}_{(t)}:=(x_1^{(t)},\dots,x_N^{(t)})^\tr$. Then $l_{(t_0)}^+$ belongs to the interior of the cone $\cC^+_{(t_0)}$. 
We claim that
\begin{align}\label{eq:l-coordinate}
\boldsymbol{x}_{(t_{i})}(\mu_{\leq i}(l_{(t_0)}^+)) = C^{\bs;t_0}_{t_i} \cdot (1, \dots, 1)^\tr
\end{align}
holds for $i=0,\dots,h-1$. It clearly holds for $i=0$. For $i >0$, from \cref{cor:x_and_E} we have
\begin{align*}
    \boldsymbol{x}_{(t_{i})}(\mu_{\leq i}(l_{(t_0)}^+))&=E_{k_{i-1},\epsilon_{i-1}}^{(t_{i-1})}\cdot\boldsymbol{x}_{(t_{i-1})}(\mu_{\leq i-1}(l_{(t_0)}^+)) \\
    &=E_{k_{i-1},\epsilon_{i-1}}^{(t_{i-1})}\cdot C^{\bs;t_0}_{t_{i-1}} \cdot (1, \dots, 1)^\tr,
\end{align*}
where $\epsilon_{i-1}:=\sgn (x_{k_{i-1}}^{(t_{i-1})}(l_{(t_0)}^+)) = \sgn(\sum_{j \in I} c_{j, k_{i-1}}^{(t_{i-1})} ) = \epsilon_{k_{i-1}}^{(t_{i-1})}$ by the induction assumption and the definition of the tropical sign. Comparing with the mutation rule \eqref{eq:mut_c-mat_2}, we see that \eqref{eq:l-coordinate} holds. 
\end{proof}
We have the following \lq\lq negative'' version of \cref{lem:pres_mat_x}:

\begin{cor}\label{cor:pres_mat_x_negative}
The PL action of $\phi$ is differentiable at any point in the interior of the non-positive cone 
\begin{align*}
    \cC^-_{(t_0)}:= \left\{ w \in \X_{(t_0)}(\bR^\trop) \mid x_i^{(t_0)}(w) \leq 0 \mbox{ for $i =1,\dots,N$} \right\},
\end{align*}
and its presentation matrix there coincides with the $C$-matrix $C^{-\bs;t_0}_{t_h}$ for the opposite seed pattern $-\bs:t \mapsto (N^{(t)},-B^{(t)})=:(N^{(-t)},B^{(-t)})$. Namely, we have
\begin{align*}
    \phi|_{\cC^-_{(t_0)}} = C^{-\bs;t_0}_{t_h}.
\end{align*}
If we denote the sign of $\phi$ at any point in $\interior \cC^-_{(t_0)}$ by $\overline{\boldsymbol{\epsilon}}_\gamma^\trop$, we have
\begin{align*}
    C^{-\bs;t_0}_{t_h} = E_\gamma^{\overline{\boldsymbol{\epsilon}}_\gamma^\trop} \quad \mbox{and} \quad 
    G^{-\bs;t_0}_{t_h} = \check{E}_\gamma^{\overline{\boldsymbol{\epsilon}}_\gamma^\trop}.
\end{align*}
\end{cor}

\begin{proof}
Let us write $\X_{(-t)}:=T_{M^{(-t)}}$ so that $\X_{-\bs} = \cup_{t \in \bT_I} \X_{(-t)}$. Then one can easily see that the monomial isomorphisms $\iota_{(t)}: \X_{(t)} \xrightarrow{\sim} \X_{(-t)}$ given by $\iota_{(t)}^*X_i^{(-t)} := (X_i^{(t)})^{-1}$ for each $t \in \bT_I$ commute with cluster transformations, and hence combine to give an isomorphism $\iota: \X_{\bs} \xrightarrow{\sim} \X_{-\bs}$. See \cite[Lemma 2.1 (b)]{FG09II}. Moreover for two vertices $t,t' \in \bT_I$, we have $t \sim_{\bs} t'$ if and only if $t \sim_{-\bs} t'$. Hence the $\bs$-equivalence class of a path in $\bT_I$ is a $(-\bs)$-equivalence class, and vice versa. The monomial isomorphism $\iota: \X_{\bs} \xrightarrow{\sim} \X_{-\bs}$ is equivariant for the action of each mutation loop. 

The tropicalized map $\iota^\trop: \X_{\bs}(\bR^\trop) \xrightarrow{\sim} \X_{-\bs}(\bR^\trop)$ is represented as $-\mathrm{Id}$ on each chart, and hence sends the non-positive cone $\cC^-_{(t_0)}$ to the non-negative cone $\cC^+_{(-t_0)}$. Therefore from \cref{lem:pres_mat_x}, we get
\begin{align*}
    \phi(w_-) = (\iota^\trop)^{-1}(\phi (\iota^\trop(w_-))) = (- \mathrm{Id})^{-1} C^{-\bs;t_0}_{t_h} (-\mathrm{Id}) w_- = C^{-\bs;t_0}_{t_h}w_-
\end{align*}
for all $w_- \in \interior\cC^-_{(t_0)}$.
The remaining assertions follows from the tropical duality \eqref{eq:tropical duality}.
\end{proof}
In particular, we have:

\begin{prop}\label{prop:sign stab ray}
Let $\gamma: t_0 \to t$ be a path which represents a mutation loop. Then for any point $w \in \interior\cC_{(t_0)}^+$ (resp. $w \in \interior\cC_{(t_0)}^-$), the path $\gamma$ is sign-stable on $\interior\cC_{(t_0)}^+$ (resp. $\interior\cC_{(t_0)}^-$) if and only if it is sign-stable on the ray $\bR_{\geq 0}w$.
\end{prop}

\begin{rem}
We can define a similar sign sequence for tropical $\A$-transformations as follows. Fix a point $v \in \A_{(t)}(\bR^\trop)$. 
Then the tropical cluster $\A$-transformation  \eqref{eq:trop a-cluster} can be written as follows:
\begin{align}\label{eq:sign a-cluster}
  a'_\indi(\mu_\indk(v)) =
\begin{cases}
    -a_\indk(v)+ \sum_{\indj \in I}[-\sgn (x_\indk(p(v))) b_{\indk\indj}]_+a_\indj(v) & \mbox{if $\indi=\indk$}, \\
    a_\indi(v) & \mbox{if $\indi \neq \indk$}.
\end{cases}
\end{align}
Indeed, it follows from $\sum_{\indj \in I}[b_{\indk\indj}]_+ a_\indj - \sum_{\indj \in I}[-b_{\indk\indj}]_+ a_\indj = \sum_{\indj \in I} b_{\indk\indj}a_\indj = (p^\trop)^*x_\indk$. 

Consider the half-spaces
\begin{align*}
    \mathcal{H}_{\indk,\epsilon}^{a,(t)}:= \{ v \in \A_{(t)}(\bR^\trop) \mid \epsilon x_\indk(p(w)) \geq 0 \}
\end{align*}
for $\indk \in I$, $\epsilon \in \{+,-\}$ and $t \in \mathbb{T}_I$. Then for an edge $\edge$, the tropical cluster $\A$-transformation $\mu_\indk: \A_{(t)}(\bR^\trop) \to \A_{(t')}(\bR^\trop)$ is differentiable at any point in $\interior \mathcal{H}_{\indk,\epsilon}^{a,(t)}$ and its presentation matrix there is $\check{E}_{\indk,\epsilon}^{(t)}$.
\end{rem}



\subsection{Perron--Frobenius property}\label{subsec:PF property}
We say that 
a path $\gamma: t_0 \xrightarrow{\mathbf{k}=(k_0,\dots,k_{h-1})} t$ in $\bT_I$ is \emph{fully-mutating} if
\begin{align*}
    \{k_0,\dots,k_{h-1}\}=I.
\end{align*}
An $\bR_{\geq 0}$-invariant set $\Omega \subset \cX_{(t_0)}(\bR^\trop) \cong \cX_\bs(\bR^\trop)$ is said to be \emph{tame} if there exists $t \in \bT_I$ such that $\Omega \cap \interior \cC^+_{(t)} \neq \emptyset$.
The following is a fundamental result on the stable presentation matrix of a sign-stable mutation loop:

\begin{thm}[Perron--Frobenius property]\label{thm:PF property}
Suppose $\gamma$ is a path as above which represents a mutation loop $\phi=[\gamma]_{\bs}$, and which is sign-stable on a tame subset $\cL$. 
Then the spectral radius of $E^{(t_0)}_{\phi,\cL}$ is attained by a positive eigenvalue $\lambda^{(t_0)}_{\phi,\cL}$, and we have $\lambda^{(t_0)}_{\phi,\cL} \geq 1$.
Moreover if either $\gamma$ is fully-mutating or $\lambda^{(t_0)}_{\phi,\cL} >1$, then one of the corresponding eigenvectors is given by the coordinate vector $\boldsymbol{x}(w_{\phi,\cL})\in \bR^I$ for some $w_{\phi,\cL} \in \cC_\gamma^\stab \setminus \{0\}$.
\end{thm}
The proof will be given in \cref{subsec:proof of PF property}. 
We have checked that the following conjecture holds true for a large number of examples, by using a computer. We do not know any counterexamples.

\begin{conj}\label{p:spec_same}
For any point $w \in \X_{(t_0)}(\bR^\trop)$ and a path $\gamma$ which represents a mutation loop, the characteristic polynomials of the matrices  $E^{\boldsymbol{\epsilon}_\gamma(w)}_\gamma$ and $\check{E}^{\boldsymbol{\epsilon}_\gamma(w)}_\gamma$ are the same up to an overall sign \footnote{The conjecture of this form is based on a suggestion by Yuma Mizuno.}. In particular, the spectral radii of these matrices are the same.
\end{conj}
Note that for an $N \times N$-matrix $E$ with $\det E=\pm 1$, the characteristic polynomials of the matrices $E$ and $\check{E}$ are the same up to an overall sign if and only if the characteristic polynomial $P_E(\nu)$ of $E$ is (anti-)palindromic: $P_E(\nu^{-1}) = \pm \nu^{-N}P_E(\nu)$. 

\begin{prop}
Suppose $\gamma: t \to t'$ is a path which represents a mutation loop.
If the exchange matrix $B^{(t)}$ is regular, then the characteristic polynomials of $E_\gamma^{\boldsymbol{\epsilon}}$ and $\check{E}_\gamma^{\boldsymbol{\epsilon}}$ are the same for any sign $\boldsymbol{\epsilon} \in \{+,-\}^{h(\gamma)}$.
In particular, \cref{p:spec_same} is true.
\end{prop}
\begin{proof}
Let $\boldsymbol{\epsilon} \in \{+,-\}^{h(\gamma)}$.
Using \cref{lem:EF_formulae} (4) repeatedly
, we obtain
\[ B^{(t)} \check{E}_\gamma^{\boldsymbol{\epsilon}} =E_\gamma^{\boldsymbol{\epsilon}} B^{(t)}. \]
The assertion follows from this equation.
\end{proof}

It turns out that the sign stability on the set $\cL^{\mathrm{can}}_{(t_0)}:=\interior\cC^+_{(t_0)} \cup \interior\cC^-_{(t_0)}$ plays a crucial role in the sequel, and the corresponding eigenvalue $\lambda^{(t_0)}_{\phi,\cL_{(t_0)}^{\mathrm{can}}}$ is a canonical numerical invariant of the mutation loop $\phi$.

\begin{dfn}[Cluster stretch factor]\label{d:stretch factor}
Suppose $\phi=[\gamma]_{\bs}$ is a mutation loop, and the path $\gamma$ is sign-stable on the set $\cL_{(t_0)}^{\mathrm{can}}$. Then we denote the stable presentation matrix (\cref{cor:asymptotic linearity}) by $E^{(t_0)}_\phi:=E^{(t_0)}_{\phi,\cL_{(t_0)}^{\mathrm{can}}}$, and we call the spectral radius $\lambda_\phi^{(t_0)}:=\lambda^{(t_0)}_{\phi,\cL_{(t_0)}^{\mathrm{can}}}$ the \emph{cluster stretch factor} of $\phi$.
\end{dfn}
Note that from the definition, the cluster stretch factor is an algebraic integer of degree at most $N$.

\begin{rem}\label{rem:intrinsicality}
If moreover the path $\gamma$ as above is fully-mutating or $\lambda^{(t_0)}_{\phi,\cL} >1$, then the cluster stretch factor $\lambda_{\phi}^{(t_0)}$ only depends on the mutation loop $\phi$. Indeed, if $\phi$ also admits a representation path $\gamma':t'_0 \to t'$ which is sign-stable on $\Omega_{(t'_0)}^{\mathrm{can}}$, then choosing a path $\delta: t_0 \to t'_0$, we get the relation
\begin{align}\label{eq:PL_conjugation}
    (d\phi_{(t'_0)})_{\mu_\delta(w)} \circ (d\mu_\delta)_w = (d\mu_\delta)_{\phi_{(t_0)}(w)} \circ (d\phi_{(t_0)})_w 
\end{align}
for any point $w \in \X_{(t_0)}(\bR^\trop)$. When $w=w_{\phi,\Omega}$ with $\Omega:=\Omega_{(t_0)}^{\mathrm{can}}$, then we have $\phi_{(t_0)}(w_{\phi,\Omega})=\lambda_{\phi}^{(t_0)}. w_{\phi,\Omega}$ and in particular the points $w_{\phi,\Omega}$ and $\phi_{(t_0)}(w_{\phi,\Omega})$ have the same sign for any path $\delta$. Denote the common presentation matrix of the PL map $\mu_\delta$ at these points by $M$. Then \eqref{eq:PL_conjugation} implies that
\begin{align*}
    E_{\phi}^{(t'_0)} = M E_{\phi}^{(t_0)} M^{-1}, 
\end{align*}
hence the spectral radii $\lambda_{\phi}^{(t_0)}$ and $\lambda_{\phi}^{(t_0)}$ are the same. In this case, we simply write $\lambda_\phi:=\lambda_\phi^{(t_0)}$.
\end{rem}

\subsection{Proof of \cref{thm:PF property}}\label{subsec:proof of PF property}

Fix a mutation loop $\phi \in \Gamma_\bs$ and its representation path $\gamma: t_0 \xrightarrow{\mathbf{k}} t$ with $h(\gamma) = h$.
For a sign $\bep \in \{+,-\}^h$ we define the cone $\cC^{\bep}_{\gamma}$ as
\[\cC^{\bep}_{\gamma} := \mathrm{rel.cl}\{ w \in \cX_{(t_0)}(\bR^\trop) \mid \bep_\gamma(w) = \bep\} \subset \cX_{(t_0)}(\bR^\trop),\]
where $\mathrm{rel.cl}$ denotes the relative closure.

\begin{prop}
The sign cone $\cC^{\boldsymbol{\epsilon}}_{\gamma}$ is polyhedral and convex.
\end{prop}
\begin{proof}
Let
\begin{align}\label{eq:gamma}
    \gamma: t_0 \overbar{k_0} t_1 \overbar{k_1} \cdots \overbar{k_{h-1}} t_h.
\end{align}
By definition, one can verify that
\begin{align}\label{eq:pb_exp_sign_cone}
\cC^{\boldsymbol{\epsilon}}_{\gamma} = \bigcap_{i=0}^{h-1} \mu_{\gamma_{\leq i}}^{-1} (\cH^{(t_{i})}_{k_{i}, \epsilon_{i}}),
\end{align}
where $\mu_{\leq 0} := \mathrm{id}$.
We will prove the claim by induction on the length of the path $\gamma$.

In the case $h(\gamma) = 1$,
the cone $\cC^{\epsilon_0}_{\gamma}$ is nothing but the half space $\cH^{(t_{0})}_{k_{0},\epsilon_0}$, so it is clearly polyhedral and convex.

In the case $h(\gamma) = h >1$,
we put $\boldsymbol{\epsilon}' := (\epsilon_0, \dots, \epsilon_{h-2})$,  $\gamma' := \gamma_{\leq h-1}$, $\gamma'' := \gamma_{\leq h-2}$ and $\cH_i := \cH^{(t_{i})}_{k_{i}, \epsilon_{i}}$ for $i= 0, \dots, h-1$.
Then,
\begin{align*}
    \cC^{\boldsymbol{\epsilon}}_{\gamma} 
    &= \cC^{\boldsymbol{\epsilon'}}_{\gamma'} \cap \mu^{-1}_{\gamma'}(\cH_{h-1})\\
    &= \mu_{\gamma''}^{-1}(\mu_{\gamma''}(\cC^{\boldsymbol{\epsilon'}}_{\gamma'}) \cap \mu_{k_{h-2}}^{-1}(\cH_{h-1}))\\
    &= \mu_{\gamma''}^{-1}(\mu_{\gamma''}(\cC^{\boldsymbol{\epsilon'}}_{\gamma'}) \cap \cH_{h-2} \cap \mu_{k_{h-2}}^{-1}(\cH_{h-1})).
\end{align*}
Here, the last equation follows from
\[
\mu_{\gamma''}(\cC^{\boldsymbol{\epsilon'}}_{\gamma'}) 
= \mu_{\gamma''}(\cH_0) \cap \mu_{\gamma''}(\mu^{-1}_{\gamma_{\leq 1}}(\cH_1)) \cap \cdots \cap \cH_{h-2}.
\]
Since $\mu_{\gamma''}$ is linear on the cone $\cC^{\boldsymbol{\epsilon'}}_{\gamma'}$, which is convex and polyhedral by the induction hypothesis, the image $\mu_{\gamma''}(\cC^{\boldsymbol{\epsilon'}}_{\gamma'})$ is also convex and polyhedral.
Since the bent locus of the boundary of $\mu_{k_{h-2}}^{-1}(\cH_{h-1})$
is contained in the boundary of $\cH_{h-2}$,
the intersection $\cH_{h-2} \cap \mu_{k_{h-2}}^{-1}(\cH_{h-1})$ is the same as the intersection of two half-spaces, so it is also convex and polyhedral.

Thus the intersection $\mu_{\gamma''}(\cC^{\boldsymbol{\epsilon'}}_{\gamma'}) \cap \cH_{h-2} \cap \mu_{k_{h-2}}^{-1}(\cH_{h-1})$ is a convex polyhedral cone, and so is the cone $\cC^{\boldsymbol{\epsilon}}_{\gamma} $ since it is the inverse image under the linear map $\mu_{k_{h-2}}|_{\cH_{h-2}}$.
\end{proof}

In order to prove \cref{thm:PF property}, we recall the notion of dual cones.
\begin{defi}[Dual cone]
Let $\cC$ be a convex cone in a finite dimensional vector space $V$.
\begin{enumerate}
    \item The dimension $\dim \cC$ of $\cC$ is the dimension of the smallest subspace of $V$ containing $\cC$.
    \item The dual cone $\cC^\vee$ of $\cC$ is a cone in the dual space $V^*$, defined by
    \[ \cC^\vee := \{ n \in V^* \mid \langle n, v \rangle \geq 0 \mbox{ for all $v \in \cC$} \}. \]
\end{enumerate}
\end{defi}
The following lemma is well-known.

\begin{lem}\label{lem:elem_cone}
\begin{enumerate}
\item A convex cone $\cC \subset V$ is strictly convex if and only if $\dim\cC^\vee = \dim V$.
\item For two convex cones $\cC_1, \cC_2 \subset V$, $(\cC_1 \cap \cC_2)^\vee = \cC_1^\vee + \cC_2^\vee$.
\end{enumerate}
\end{lem}



\begin{prop}\label{prop:str_conv}
The sign cone $\cC^{\boldsymbol{\epsilon}}_{\gamma}$ is strictly convex if $\gamma$ is fully-mutating.
\end{prop}
\begin{proof}
Let $\gamma$ be as in \eqref{eq:gamma}.
For each $i=0, 1, \dots, h-1$, we have
\begin{align*}
\mu_{\gamma_{\leq i}}^{-1} (\cH^{(t_{i})}_{k_{i}, \epsilon_{i}})
&= \{ \mu_{\gamma_{\leq i}}^{-1}(w') \mid w' \in \X_{(t_i)}(\bR^\trop),\ \langle \epsilon_i \cdot e^{(t_{i})}_{k_{i}}, w' \rangle \geq 0 \}\\
&= \{ w \in \X_{(t_0)}(\bR^\trop) \mid \langle \epsilon_i \cdot e^{(t_{i})}_{k_{i}}, \mu_{\gamma_{\leq i}}(w) \rangle \geq 0 \}\\
&= \{ w \in \X_{(t_0)}(\bR^\trop) \mid \langle \epsilon_\nu \cdot e^{(t_{i})}_{k_{i}}, E_{\gamma_{\leq i}}^{\boldsymbol{\epsilon}_{\leq i}} \cdot w \rangle \geq 0 \}\\
&= \{ w \in \X_{(t_0)}(\bR^\trop) \mid \langle \epsilon_i \cdot (E_{\gamma_{\leq i}}^{\boldsymbol{\epsilon}_{\leq i}})^\tr \cdot e^{(t_{i})}_{k_{i}}, w \rangle \geq 0 \},
\end{align*}
where $\boldsymbol{\epsilon}_{\leq i} = (\epsilon_0, \dots, \epsilon_{i-1})$.
Since $(e^{(t_{i})}_{j})_j$ is the basis of $N^{(t_{i})}$,
$\mathbf{c}_i := (E_{\gamma_{\leq i}}^{\boldsymbol{\epsilon}_{\leq i}})^\tr \cdot e^{(t_{i})}_{k_{i}}$ is the $k_{i}$-th row vector of $E_{\gamma_{\leq i}}^{\boldsymbol{\epsilon}_{\leq i}}$.
Thus,
\begin{align*}
\hspace{3cm}
(\cC^{\boldsymbol{\epsilon}}_\gamma)^\vee 
&= \sum_{i=0}^{h-1} \Big( \mu_{\gamma_{\leq i}}^{-1}(\cH^{(t_{i})}_{k_{i}, \epsilon_{i}}) \Big)^\vee \qquad (\mbox{by \eqref{eq:pb_exp_sign_cone} and \cref{lem:elem_cone} (2)})\\
&= \sum_{i=0}^{h-1} \bR_{\geq 0} \epsilon_i \cdot \mathbf{c}_i
\end{align*}
Now we claim that 
\begin{align*}
    \mathrm{span}_{\bR} \{ \mathbf{c}_\nu\}_{\nu=0}^j = \mathrm{span}_{\bR} \{ e_{k_\nu}^{(t_0)}\}_{\nu=0}^j
\end{align*}
holds for $0 \leq j \leq h-1$. 
Indeed, with a notice that the matrices $E_{k, \epsilon}^{(t')}$ are row operator matrices (except for the $-1$ at the $(k,k)$-entry), one can easily see that the inclusion \lq\lq $\subseteq$" holds. To prove the converse inclusion \lq\lq $\supseteq$", we proceed by induction on $j$. For $j=0$, it is obvious since $e_{k_0}^{(t_0)}=\mathbf{c}_0$. Assume the $j$-th step of the induction. 
If $k_{j+1} \in \{k_0,\dots,k_j\}$, then the claim is also obvious. If $k_{j+1} \notin \{k_0,\dots,k_j\}$, then the matrix $E_{\gamma_{\leq j+1}}^{\boldsymbol{\epsilon}_{\leq j+1}}$ does not include the row operation on the $k_{j+1}$-th row,
and hence its $k_{j+1}$-th column vector is still the basis vector. Hence 
\begin{align*}
    \mathbf{c}_{j+1}-e_{k_{j+1}}^{(t_0)} \in  \mathrm{span}_{\bR} \{ e_{k_\nu}^{(t_0)}
    \}_{\nu=0}^j = \mathrm{span}_{\bR} \{ \mathbf{c}_\nu \}_{\nu=0}^j
\end{align*}
by the fact that $\mathbf{c}_{j+1} \in \mathrm{span}_\bR\{ e_{k_\nu}^{(t_0)}\}_{\nu=0}^{j+1}$ which we have confirmed and the induction hypothesis. Thus the claim is proved. 
Hence $\dim ((\cC^{\boldsymbol{\epsilon}}_\gamma)^\vee) = N$ by the fully-mutating condition, which implies the assertion by \cref{lem:elem_cone} (1).
\end{proof}

\begin{proof}[Proof of \cref{thm:PF property}]
We first prove that the cone
\begin{align*}
    \cC_\gamma^\stab:= \bigcap_{n \geq 1} \cC_{\gamma^n}^{(\bep_\gamma^\stab)^n},
\end{align*}
is $N$-dimensional.
Since $\Omega$ is tame, $\phi$ is sign-stable on $\cC^+_{(t)}$ for some vertex $t \in \bT_I$ by \cref{cor:C_as_pres_mat}.
Then, there exists $n_0 >0$ such that $\phi^{n_0} (\cC) \subset \cC^\stab_\gamma$, where $\cC \subset \cX_{(t_0)}(\bR^\trop)$ is the cone corresponding to $\cC_{(t)}^+$ by the coordinate transformation $\cX_{(t_0)}(\bR^\trop) \cong \cX_{(t)}(\bR^\trop)$.
Thus $\cC^\stab_\gamma$ is $N$-dimensional since $\phi$ is linear isomorphism on $\cC$ again by \cref{cor:C_as_pres_mat}.

Next we consider the case where the path $\gamma$ is fully-mutating. In this case, by \cref{prop:str_conv} the sign cone $\cC_\gamma^{\bep_\gamma^\stab}$ is strictly convex, and so is the cone
\begin{align*}
    \cC_\gamma^\stab:= \bigcap_{n \geq 1} \cC_{\gamma^n}^{(\bep_\gamma^\stab)^n},
\end{align*}
which is also $E_\phi$-stable by definition.
Here, $(\bep_\gamma^\stab)^n := (\underbrace{\bep_\gamma^\stab, \dots, \bep_\gamma^\stab}_{n})$.
Then by \cite[Chapter 1, Theorem 3.2]{BP}, the spectral radius $\rho(E^{(t_0)}_\phi)$ is an eigenvalue and the cone $\cC_\gamma^{\bep_\gamma^\stab}$ contains a corresponding eigenvector $w_{\phi,\cL}$. Since $\det E_\phi^{(t_0)}=\pm 1$, we have $\rho(E_\phi^{(t_0)}) \geq 1$ and the assertion is proved. 

Let us consider the case where the spectral radius $\rho(E_\phi^{(t_0)})>1$, and $\gamma$ is not necessarily fully-mutating. Write $\gamma$ as in \eqref{eq:gamma}. 
Let $I(\gamma):=\{k_0,\dots,k_{h-1}\} \subset I$, and $\overline{\bs}$ be the seed pattern with the seed 
\begin{align*}
    \Big(N^{(t_0)}|_{I(\gamma)}:=\bigoplus_{i \in I(\gamma)} \bZ e_i^{(t_0)}, B^{(t_0)}|_{I(\gamma)}:=(b_{ij}^{(t)})_{i,j \in I(\gamma)}\Big)
\end{align*}
at the initial vertex $t_0$. Then $\gamma$ can be regarded as an edge path in $\bT_{I(\gamma)}$, which is fully-mutating and represents a mutation loop $\overline{\phi}:=[\gamma]_{\overline{\bs}}$. 
From the form of the matrices $E_{k,\epsilon}^{(t')}$ (recall \cref{rem:elementary matrix}), we have the block-decomposition
\begin{align*}
    E_\phi^{(t_0)}=
    \begin{pmatrix}
    E_{\overline{\phi}}^{(t_0)} & 0\\
    X_\phi^{(t_0)} & 1
    \end{pmatrix}
\end{align*}
with respect to the direct sum decomposition $\X_{(t_0)}(\bR^\trop)=\bR^I = \bR^{I(\gamma)} \oplus \bR^{I \setminus I(\gamma)}$ for some $|I\setminus I(\gamma)| \times |I(\gamma)|$-matrix $X_\phi^{(t_0)}$ and the stable presentation matrix $E^{(t_0)}_{\overline{\phi}}$ of the mutation loop $\overline{\phi}$.  
In particular, $\rho(E_\phi^{(t_0)}) = \rho(E_{\overline{\phi}}^{(t_0)})$. Then by applying the argument above to $E^{(t_0)}_{\overline{\phi}}$, we see that $\rho(E^{(t_0)}_{\overline{\phi}})$ is an eigenvalue of $E^{(t_0)}_{\overline{\phi}}$, hence of $E^{(t_0)}_\phi$. One of the corresponding eigenvector of the latter is given by
\begin{align*}
    w_{\phi,\cL}:=\left(w_{\overline{\phi},\overline{\cL}}, \frac{X^{(t_0)}_\phi w_{\overline{\phi},\overline{\cL}}}{\lambda^{(t_0)}_\phi-1}\right),
\end{align*}
where $\overline{\cL} \subset \X_{\overline{\bs}}(\bR^\trop)$ is the image of $\cL$ under the projection $\pi: \X_{\bs}(\bR^\trop) \to \X_{\overline{\bs}}(\bR^\trop)$ determined by $\pi^* X_i^{(t_0)}:=X_i^{(t_0)}$ for $i \in I(\gamma)$. Thus the theorem is proved.
\end{proof}

\begin{rem}\label{rem:power_convexity}
From the proof, the statements in \cref{thm:PF property} still holds if the cone $\cC_{\gamma^n}^{(\boldsymbol{\epsilon}_\gamma^\stab)^n}$ is strictly convex for some $n \geq 1$. This generalization is useful when we deal with a general mutation loop including permutations of indices. See \cref{rem:FMloop_convexity}.
\end{rem}
\section{Algebraic entropy of cluster transformations}
Let us first recall the definition of the algebraic entropy following \cite{BV99}. 

For a rational function $f(u_1,\dots,u_N)$ over $\bQ$ on $N$ variables, write it as $f(u)=g(u)/h(u)$ for two polynomials $g$ and $h$ without common factors. Then the \emph{degree} of $f$, denoted by $\deg f$, is defined to be the maximum of the degrees of the constituent polynomials $g$ and $h$. 
For a homomorphism $\varphi^*:\bQ(u_1,\dots,u_N) \to \bQ(u_1,\dots,u_N)$ between the field of rational functions on $N$ variables, let $\varphi_i:=\varphi^*(u_i)$ for $i=1,\dots,N$. Since $\bQ(u_1,\dots,u_N)$ is the field of rational functions on the algebraic torus $\bG_m^N$ equipped with coordinate functions $u_1,\dots,u_N$, the homomorphism $\varphi^*$ can be regarded as the pull-back action via a rational map $\varphi:\bG_m^N \to \bG_m^N$ between algebraic tori. 
We define the degree of $\varphi$ to be the maximum of the degrees  $\deg\varphi_1,\dots,\deg\varphi_N$ and denote it by $\deg (\varphi)$. 
\begin{dfn}[Bellon--Viallet \cite{BV99}]\label{d:entropy}
The \emph{algebraic entropy} $\cE_\varphi$ of a rational map $\varphi:\bG_m^N \to \bG_m^N$ is defined as
\begin{align*}
    \cE_\varphi:= \limsup_{n \to \infty} \frac{1}{n}\log(\deg(\varphi^n)).
\end{align*}
\end{dfn}
Since $\deg(\varphi^n) \leq \deg(\varphi)^n$, the algebraic entropy is always finite. Here are basic properties:
\begin{itemize}
    \item For any rational map $\varphi: \bG_m^N \to \bG_m^N$ and an integer $m \geq 0$, we have 
    \begin{align}\label{eq:entropy_power}
        \cE_{\varphi^m} = m\cE_{\varphi}.
    \end{align}
    \item The algebraic entropy is conjugation-invariant. Namely, we have 
    \begin{align}\label{eq:entropy_conjugation}
        \cE_{f\varphi f^{-1}} =\cE_\varphi
    \end{align}
    for a rational map $\varphi: \bG_m^N \to \bG_m^N$ and a birational map $f: \bG_m^N \to \bG_m^N$.
\end{itemize}

Our aim is to compute the algebraic entropies of cluster transformations induced by a sign-stable mutation loop. 
For a mutation loop $\phi$, let us write $\cE_\phi^z:=\cE_{\phi^z}$ for $z=a,x$. 
Here is our main theorem:

\begin{thm}\label{t:entropy}
Let $\phi=[\gamma]_{\bs}$ be a mutation loop with a representation path $\gamma: t_0 \to t_h$ which is sign-stable on the set $\cL_{(t_0)}^{\mathrm{can}}$. Then we have
\begin{align*}
    \log \rho(\check{E}^{(t_0)}_\phi) &\leq \cE_\phi^a \leq \log R^{(t_0)}_\phi,\\
    \log \rho(E^{(t_0)}_\phi) &\leq \cE_\phi^x \leq \log R^{(t_0)}_\phi.
\end{align*}
Here $R^{(t_0)}_\phi:=\max\{\rho(E^{(t_0)}_\phi),\rho(\check{E}^{(t_0)}_\phi)\}$.
\end{thm}  
Note that if \cref{p:spec_same} holds true, then we get $R_\phi^{(t_0)}=\rho(\check{E}_\phi^{(t_0)})=\rho(E_\phi^{(t_0)})=\lambda_\phi^{(t_0)}$. In particular we obtain \cref{intro: main cor}. 

Before proceeding to the proof, let us prepare some notations. Recall that for a representation path $\gamma: t_0 \xrightarrow{\mathbf{k}} t_h$ of the mutation loop $\phi$, the path $\gamma^n:t_0 \xrightarrow{\mathbf{k}} t_h \xrightarrow{\mathbf{k}} \dots \xrightarrow{\mathbf{k}} t_{nh}$ represents the mutation loop $\phi^n$. We denote the data attached to the vertex $t_{nh}$ with a superscript $(n)$. For instance, $A_\indi^{(n)}:=A_\indi^{(t_{nh})}$, $X_\indi^{(n)}:=X_\indi^{(t_{nh})}$, and so on. 
We simply write $E_\phi:=E_\phi^{(t_0)}$ and $R_\phi:=R_\phi^{(t_0)}$. 

\subsection{Estimate of the entropy from below}\label{sec:from below}
For an $N\times N$-matrix $M=(m_{ij})_{i,j=1,\dots,N}$, let $\|M\|_{\max}:=\max_{i,j=1,\dots,N} \{|m_{ij}|\}$ denote the uniform norm.
On the other hand, since any two norms on $\bR^N$ are equivalent, there exists a universal constant $K >0$ such that $K^{-1} \|\cdot \|_\infty \leq \|\cdot\|_1 \leq K \|\cdot \|_\infty$ on $\bR^N$. 

\begin{lem}\label{l:frombelow_1} 
For any $n \geq 0$, we have $\degr((\phi^x)^n) \geq K^{-1} \|C^{(n)}\|_{\max}$.
\end{lem}

\begin{proof}
Recall the separation formula \eqref{eq:separation_X}. With the help of \cref{l:no cancellation}, we get 
\begin{align*}
    \degr(X_\indi^{(n)}) \geq \degr(X_1^{c_{\indi 1}^{(n)}}\dots X_N^{c_{\indi N}^{(n)}}) = \| \mathbf{c}_\indi^{(n)}\|_1 \geq K^{-1} \|\mathbf{c}_\indi^{(n)}\|_\infty. 
\end{align*}
Hence we have
\begin{align*}
    \degr ((\phi^x)^n)=\max_{\indi=1,\dots,N} \degr(X_\indi^{(n)}) \geq K^{-1} \max_{\indi=1,\dots,N}\|\mathbf{c}_\indi^{(n)}\|_\infty = K^{-1}\|C^{(n)}\|_{\max}.
\end{align*}
\end{proof}

Let $\overline{C}^{(n)}$ (resp. $\overline{G}^{(n)}$) denote the $C$-matrix (resp. $G$-matrix) associated to the seed pattern $-\bs: t \mapsto (N^{(t)}, -B^{(t)})$.

\begin{lem}
For any $n \geq 0$, we have $\degr((\phi^a)^n) \geq    \|\overline{G}^{(n)}\|_{\max} /2$.
\end{lem}
\begin{proof}
Here note that the separation formula \eqref{eq:separation_A} itself does not give a reduced expression for the rational function $A_i^{(t)}$. 
For $t \in \bT_I$ and $i \in I$, we have
\begin{align*}
    F_i^{(t)} = \sum_{ \mathbf{m} \in M_i^{(t)}} y_k^{m_k},
\end{align*}
where the set $M_i^{(t)}$ consists of integer vectors $\mathbf{m}=(m_1,\dots,m_N)$ with non-negative coordinates such that $0 \leq m_j \leq f_{ij}^{(t)}$ for all $j \in I$.
Note that $\mathbf{f}_i^{(t)} \in M_i^{(t)}$. 
Using $\delta_{ij}^{(t)}:=\max \bigl\{[-\sum_k m_k b_{kj}^{(t)}]_+ \ \bigl| \  \mathbf{m} \in M_i^{(t)} \bigr\}$, 
the reduced expression of $F_i^{(t)}(p^*X_1, \dots, p^*X_N)$ can be written as
\begin{align*}
    F_i^{(t)}(p^*X_1, \dots, p^*X_N) = \frac{\sum_{\mathbf{m} \in M_i^{(t)}} \prod_j A_j^{ \sum_k m_k b_{kj}^{(t)}+\delta_{ij}^{(t)}} }{\prod_j A_j^{\delta_{ij}^{(t)}}}.
\end{align*}
Then the expression 
\begin{align*}
    A_i^{(t)} = \prod_j\frac{A_j^{[g_{ij}^{(t)}]_+}}{A_j^{[-g_{ij}^{(t)}]_+}} \cdot \frac{\sum_{\mathbf{m} \in M_i^{(t)}} \prod_j A_j^{ \sum_k m_k b_{kj}^{(t)}+\delta_{ij}^{(t)}} }{\prod_j A_j^{\delta_{ij}^{(t)}}},
\end{align*}
which is obtained from \eqref{eq:separation_A}, may fail to be a reduced expression only for the reason that the numerator $\prod_j A_j^{[g_{ij}^{(t)}]_+}$ may have a common factor with the denominator $\prod_j A_j^{\delta_{ij}^{(t)}}$. 
For each $j \in I$, let $\deg_{A_j}$ denote the degree as a rational function of $A_j$, other variables being regarded as coefficients. 
\begin{description}
    \item[The case $\delta_{ij}^{(t)} = 0$] 
    In this case, the monomial $A_j^{g^{(t)}_{ij}}$ has no common factors with other terms. When $g_{ij}^{(t)} \geq 0$,  
    \begin{align*}
        \deg_{A_j}(A_i^{(t)}) 
        &\geq g^{(t)}_{ij} + \sum_k f^{(t)}_{ik}b^{(t)}_{kj}=\Bigl|g^{(t)}_{ij} + \sum_k f^{(t)}_{ik}b^{(t)}_{kj}\Bigr|. 
    \end{align*}
    When $g_{ij}^{(t)} \leq 0$,
    \begin{align*}
        \deg_{A_j}(A_i^{(t)}) 
        &\geq \max\Bigl\{|g^{(t)}_{ij}|,\ \sum_k f^{(t)}_{ik}b^{(t)}_{kj}\Bigr\} \\
        &\geq \frac{1}{2}\Bigl(|g^{(t)}_{ij}|+ \Bigl|\sum_k f^{(t)}_{ik}b^{(t)}_{kj}\Bigr|\Bigr) 
        \geq \frac{1}{2}\Bigl|g^{(t)}_{ij} + \sum_k f^{(t)}_{ik}b^{(t)}_{kj}\Bigr|. 
    \end{align*}
    \item[The case $\delta_{ij}^{(t)} > 0$]
    When $g_{ij} \leq 0$, still there cannot be a cancellation, and hence
    \begin{align*}
         \deg_{A_j}(A_i^{(t)}) 
        &\geq \max\Bigl\{|g^{(t)}_{ij}-\delta_{ij}^{(t)}|,\ \Bigl|\sum_k f^{(t)}_{ik}b^{(t)}_{kj}+\delta_{ij}^{(t)}\Bigr| \Bigr\} \\
        &\geq \frac{1}{2}\Bigl(|g^{(t)}_{ij}-\delta_{ij}^{(t)}|+ \Bigl|\sum_k f^{(t)}_{ik}b^{(t)}_{kj} + \delta_{ij}^{(t)}\Bigr| \Bigr) 
        \geq \frac{1}{2}\Bigl|g^{(t)}_{ij} + \sum_k f^{(t)}_{ik}b^{(t)}_{kj} \Bigr|. 
    \end{align*}
    When $g_{ij} \geq 0$, a reduction can occur but still we have
    \begin{align*}
        \deg_{A_j}(A_i^{(t)}) 
        &\geq \begin{dcases}
        (g^{(t)}_{ij}-\delta_{ij}^{(t)})+ \Bigl(\sum_k f^{(t)}_{ik}b^{(t)}_{kj}+\delta_{ij}^{(t)} \Bigr) & \mbox{if $g^{(t)}_{ij}-\delta_{ij}^{(t)} \geq 0$}, \\
        \max\Bigl\{|g^{(t)}_{ij}-\delta_{ij}^{(t)}|,\ \sum_k f^{(t)}_{ik}b^{(t)}_{kj}+\delta_{ij}^{(t)}\Bigr\} & \mbox{if $g^{(t)}_{ij}-\delta_{ij}^{(t)} \leq 0$}
        \end{dcases} \\
        &\geq \frac{1}{2}\Bigl(|g^{(t)}_{ij}-\delta_{ij}^{(t)}|+ \Bigl|\sum_k f^{(t)}_{ik}b^{(t)}_{kj}+\delta_{ij}^{(t)}\Bigr|\Bigr) 
        \geq \frac{1}{2}\Bigl|g^{(t)}_{ij} + \sum_k f^{(t)}_{ik}b^{(t)}_{kj}\Bigr|.
    \end{align*}
\end{description}
By summarizing above inequalities and applying to the case $t = t_{nh}$, we get
\[\deg_{A_j}(A_i^{(n)}) \geq \frac{1}{2}\Bigl|g^{(n)}_{ij} + \sum_k f^{(n)}_{ik}b^{(n)}_{kj}\Bigr|.\]
Thus we have
\begin{align*}
    \deg ((\phi^a)^n)= \max_{i\in I}\deg(A_i^{(n)}) 
    \geq \max_{i,j \in I}\bigl\{ \deg_{A_j}(A^{(n)}_i)\bigr\}
    \geq \frac{1}{2}\bigl\| G^{(n)} + F^{(n)}B^{(n)} \bigr\|_{\max} = \frac{1}{2}\bigl\| \overline{G}^{(n)} \bigr\|_{\max}.
\end{align*}
Here the last equation follows from the second equation given in \cref{t:CGF_of_-bs}.
\end{proof}

\begin{rem}
A similar estimate of $\deg ((\phi^a)^n)$ using $D$-matrices (which have $d$-vectors as column vectors) might be easier. However we do not know if the linear independence of $d$-vectors holds in general, which we would need in the proof of \cref{p:estimate_below} below.
\end{rem}

\begin{lem}\label{l:frombelow_2}
Suppose that the representation path $\gamma$ of $\phi$ is sign-stable on $\cL_{(t_0)}^{\mathrm{can}}$. Then there exists an integer $n_0 \geq 0$ such that
\begin{align*}
    C^{(n+1)}&=E_\phi C^{(n)}, \quad \overline{C}^{(n+1)}=E_\phi \overline{C}^{(n)}, \\ 
    G^{(n+1)}&=\check{E}_\phi G^{(n)}, \quad \overline{G}^{(n+1)}=\check{E}_\phi \overline{G}^{(n)}
\end{align*}
for all $n \geq n_0$.
\end{lem}

\begin{proof}
From \cref{cor:C_as_pres_mat}, the $C$-matrix $C^{(n)}$ assigned to the endpoint of the $n$-th iterated path $\gamma^n$ is given by the matrix $E_\gamma^{\boldsymbol{\epsilon}_{\gamma^n}^\trop}$. Moreover, the sequence $\boldsymbol{\epsilon}_{\gamma^n}^\trop$ of  tropical signs coincides with the sign at any point $w$ in $\interior\cC^+_{(t_0)}$. In particular, we have the recurrence relation $C^{(n+1)}= E_\gamma^{\boldsymbol{\epsilon}_\gamma(\phi^n(w))} C^{(n)}$ for all $n \geq 0$. Since $\gamma$ is assumed to be sign-stable on $\interior\cC^+_{(t_0)}$, there exists an integer $n_1 \geq 0$ such that $\boldsymbol{\epsilon}_\gamma(\phi^n(w))=\boldsymbol{\epsilon}_\gamma^\stab$ for all $n \geq n_1$. Thus we have $C^{(n+1)}= E_\phi C^{(n)}$ for all $n \geq n_1$. 

Similarly from \cref{cor:pres_mat_x_negative}, we get  $\overline{C}^{(n+1)}= E_\gamma^{\boldsymbol{\epsilon}_\gamma(\phi^n(w_-))} \overline{C}^{(n)}$ for any $w_- \in \interior\cC^-_{(t_0)}$. Since $\phi$ is also  sign-stable on $\interior\cC^-_{(t_0)}$, there exists an integer $n_2 \geq 0$ such that $\boldsymbol{\epsilon}_\gamma(\phi^n(w_-))=\boldsymbol{\epsilon}_\gamma^\stab$ and $\overline{C}^{(n+1)}= E_\phi \overline{C}^{(n)}$ for $n \geq n_2$. Putting $n_0:=\max\{n_1,n_2\}$, we get the desired assertion for $C$-matrices. The proof of the assertions for $G$-matrices follows from the same line of arguments.
\end{proof}
Combining \cref{l:frombelow_1,l:frombelow_2}, we get an estimate of the algebraic entropies $\cE_\phi^x$ and $\cE_\phi^a$ from below:

\begin{prop}\label{p:estimate_below}
We have $\cE_\phi^x \geq \log \rho(E_\phi)$ and $\cE_\phi^a \geq \log \rho(\check{E}_\phi)$.
\end{prop}

\begin{proof}
From \cref{l:frombelow_2}, there exists $n_0 \geq 0$ such that $C^{(n)} = E_\phi^{n-n_0}C^{(n_0)}$. Combining with the first estimate in \cref{l:frombelow_1}, we get
\begin{align*}
    \cE_\phi^x 
    &= \limsup_{n \to \infty} \frac{1}{n}\log \degr (\phi^x)^n \\
    &\geq \limsup_{n \to \infty} \left(\frac{1}{n}\log  \bigl\|C^{(n)}\bigr\|_{\max} + \frac{1}{n}\log K^{-1} \right) \\
    &= \limsup_{n \to \infty} \frac{1}{n}\log \bigl\|E_\phi^{n-n_0}C^{(n_0)}\bigr\|_{\max} \\ 
    &= \mathsf{\Lambda}_{E_\phi}(C^{(n_0)}).
\end{align*}
Here $\mathsf{\Lambda}_{E_\phi}(C^{(n_0)}):=\max_{j=1,\dots,N} \mathsf{\Lambda}_{E_\phi}(\widetilde{\mathbf{c}}_j^{(n_0)})$ with $C^{(n_0)}=(\widetilde{\mathbf{c}}_1^{(n_0)},\dots,\widetilde{\mathbf{c}}_N^{(n_0)})$, and $\mathsf{\Lambda}_{E_\phi}(-)$ denotes the Lyapunov exponent. See \cref{subsec:Lyapunov}. 
Since $C^{(n_0)}$ is invertible, the column vectors $\widetilde{\mathbf{c}}_j^{(n_0)}$ are linearly independent. Hence the maximum of their Lyapunov exponents attains the logarithm of the spectral radius $\rho(E_\phi)$ from \cref{cor:Lyapunov_independent}. Thus we get $\cE_\phi^x \geq \log \rho(E_\phi)$. The proof of the second estimate follows from the same line of arguments.
\end{proof}

\subsection{Estimate of the entropy from above}\label{sec:from above}

\begin{lem}\label{l:fromabove_1}
Let $K':=KN\max_{\indi, \indj} |b_{\indj \indi}| \geq K >0$. Then for any $n \geq 0$, we have
\begin{align*}
    \degr((\phi^x)^n) &\leq K' (\|C^{(n)}\|_{\max} + \|F^{(n)}\|_{\max}), \\
    \degr((\phi^a)^n) &\leq K' (\|G^{(n)}\|_{\max} + \|F^{(n)}\|_{\max}). 
\end{align*}
\end{lem}

\begin{proof}
From the separation formula \eqref{eq:separation_X}, we get
\begin{align*}
    \degr X_\indi^{(n)} 
    &\leq \degr(X_1^{c_{\indi 1}^{(n)}}\cdots X_N^{c_{\indi N}^{(n)}}) + \sum_{\indj=1}^N |b_{\indj \indi}| \degr(F_\indj^{(n)}) \\
    &\leq \|\mathbf{c}_\indi^{(n)}\|_1 + \max_{\indk} |b_{\indk \indi}|\cdot \sum_{\indj=1}^N \|\mathbf{f}_\indj^{(n)}\|_1 \\
    &\leq K\|\mathbf{c}_\indi^{(n)}\|_\infty + \max_{\indk} |b_{\indk \indi}|\cdot \sum_{\indj=1}^N K\|\mathbf{f}_\indj^{(n)}\|_\infty \\
    &\leq K'\max_{\indj=1,\dots,N}\|\mathbf{c}_\indj^{(n)}\|_\infty + K' \max_{\indj=1,\dots,N} \|\mathbf{f}_\indj^{(n)}\|_\infty \\
    &= K'(\|C^{(n)}\|_{\max} + \|F^{(n)}\|_{\max}).
\end{align*}
Similarly from the separation formula \eqref{eq:separation_A}, we get 
\begin{align*}
    \degr A_\indi^{(n)} 
    &\leq \degr(A_1^{g_{\indi 1}^{(n)}}\cdots A_N^{g_{\indi N}^{(n)}}) +  \degr(F_\indi^{(n)})\deg (p_{(t_{nh})}) \\
    &\leq \|\mathbf{g}_\indi^{(n)}\|_1 + \|\mathbf{f}_\indi^{(n)}\|_1\max_{\indk} \sum_{\indj}|b_{\indk \indj}|  \\
    &\leq K\|\mathbf{g}_\indi^{(n)}\|_\infty + K' \|\mathbf{f}_\indi^{(n)}\|_\infty \\
    &\leq K'(\|G^{(n)}\|_{\max} + \|F^{(n)}\|_{\max}).
\end{align*}
Thus we get the desired assertion.
\end{proof}
Now we only need to give an estimate for the growth of  $\|F^{(n)}\|_\infty$ from above. 

\begin{lem}\label{l:F-matrix_mutation_stable}
Suppose that the representation path $\gamma$ of $\phi$ is sign-stable on $\cL_{(t_0)}^{\mathrm{can}}$. Then there exists an integer $n_0 \geq 0$ such that 
\begin{align*}
    F^{(n+1)} = \check{E_\phi}F^{(n)} + \left( \sum_{m=0}^{h-1} \check{E}_h \cdots \check{E}_{h-m} \cdot \epsilon_{h-m-1} \Delta_{k_{h-m-1}} \cdot E_{h-m-2}\cdots E_{-1} \right) \overline{C}^{(n)}
\end{align*}
for all $n \geq n_0$. Here
\begin{itemize}
    \setlength{\itemsep}{3pt}
    \item $\mathbf{k} = (k_0, \dots, k_{h-1})$,
    \item $\boldsymbol{\epsilon}_\gamma^\stab=(\epsilon_1,\dots,\epsilon_h)$ is the stable sign,
    \item $E_r = E_{k_r, \epsilon_r}^{(t_{nh+r})}$ and $\check{E}_r = \check{E}_{k_r, \epsilon_r}^{(t_{nh+r})}$ for $r = 0, \dots, h-1$,
    \item $E_{-1} = \check{E}_h := \mathrm{Id}$, and
    \item $\Delta_k := \mathrm{diag}(0, \dots, 0, \overset{k}{1}, 0, \dots, 0)$.
\end{itemize}

\end{lem}

\begin{proof}
This is a consequence of an iterated application of \cref{t:FuGy}. Here note that the sign stability on the set $\cL_{(t_0)}^{\mathrm{can}}$ implies that the two sign sequences $\boldsymbol{\epsilon}_{\gamma}(\phi^n(w))$ and $\boldsymbol{\epsilon}_{\gamma}(\phi^n(w_-))$ for $w \in \interior \cC^+_{(t_0)}$ and $w_- \in \interior \cC^-_{(t_0)}$ stabilize to the same stable sign $\boldsymbol{\epsilon}_\gamma^\stab$ for large $n$. Since the former one is the tropical sign and the latter is the sign of the row vectors of $C_t^{-\bs;t_0}$ (cf. the definition of the tropical sign), we see that the second term of the equation given in \cref{t:FuGy} becomes $(C_t^{-\bs;t_0})^{k\bullet}$ in the stable range. With a notice that 
\begin{align*}
    C^{-\bs;t_0}_{t_{nh+h-m-1}} = E_{h-m-2} \cdots E_{0} E_{-1}\, C^{-\bs;t_0}_{t_{nh}},
\end{align*}
the assertion follows from a direct computation.
\end{proof}

A \lq\lq rotated'' uniform norm is suited for our computation in the sequel.  
For a real invertible matrix $E \in GL_N(\bR)$, consider its real Jordan normal form: $S^{-1}ES = J(\nu_1,m_1) \oplus \cdots \oplus J(\nu_r,m_r)$
for some real invertible matrix $S \in GL_N(\bR)$. See \cref{subsec:real Jordan}. Let $\mathbf{e}_1,\dots,\mathbf{e}_N$ be the corresponding real Jordan basis. Then for a vector $v=x_1\mathbf{e}_1+\dots+ x_N\mathbf{e}_N \in \bR^N$, we define $\|v\|_\infty^E:=\max_{k=1,\dots,N} |x_k|$. Then clearly $\|\cdot\|_\infty^E$ defines a norm on $\bR^N$, which has the following nice property:

\begin{lem}\label{l:spec_radius}  Suppose $\det E =\pm 1$.
Then for any $v \in \bR^N$, we have the inequalities $\|Ev\|_\infty^E \leq \rho(E)\|v\|_\infty^E$ and $\|\check{E}v\|_\infty^E \leq \rho(\check{E})\|v\|_\infty^E$. Here recall \cref{conv:check}.
\end{lem}

\begin{proof}
The first statement follows from the inequality  $\|J(\nu,m)\|_\infty \leq \max(|\nu|,1)$ for a real Jordan block $J(\nu,m)$, and the fact that $\rho(E) \geq 1$. The second statement follows from $S^{-1}\check{E}S = \check{J}(\nu_1,m_1) \oplus \cdots \oplus \check{J}(\nu_r,m_r)$ and $\|\check{J}(\nu,m)\|_\infty=\|J(\nu,m)\|_\infty$.
\end{proof}
For a matrix $M=(m_{ij})_{i,j=1}^N$ with column vectors $\mathbf{m}_j:=(m_{ij})_{i=1}^N$, let $\|M\|_{\max}^E := \max_{j=1,\dots,N} \|\mathbf{m}_j\|_\infty^E$. We also use the operator norm 
\begin{align*}
    \|M\|_{\mathrm{op}}^E:= \sup_{v \in \bR^N\setminus \{0\}}\frac{\|Mv\|_\infty^E}{\|v\|_\infty^E}.
\end{align*}

We are going to use these norms for $E=E_\phi$. 

\begin{lem}\label{l:fromabove_2}
Let the notations as in \cref{l:F-matrix_mutation_stable}. 
Let 
\begin{align*}
    \const:=h\max_{m=0,\dots,h-1} \bigl\|\check{E}_h \cdots \check{E}_{h-m} \cdot \epsilon_{h-m-1} \Delta_{k_{h-m-1}} \cdot E_{h-m-2}\cdots E_{-1} \bigr\|_{\mathrm{op}}^{E_\phi}.
\end{align*}
and $\constt:=\const \rho(E_\phi)^{-n_0-1}\big\|C^{(n_0)}_{-\bs}\big\|_{\max}^{E_\phi}$.
\vspace{1mm}Let $(s_n)_{n \geq n_0}$ be the sequence such that $s_{n+1}=s_n + \constt$ and $s_{n_0}=\rho(E_\phi)^{-n_0}\big\|F^{(n_0)}\big\|_{\max}^{E_\phi}$. Then we have
\begin{align*}
    \bigl\|F^{(n)}\bigr\|_{\max}^{E_\phi} \leq s_n R_\phi^n
\end{align*}
for all $n \geq n_0$.
\end{lem}

\begin{proof}
The initial condition for the sequence $(s_n)$ is chosen so that the assertion is true for $n=n_0$. Let us proceed by induction on $n \geq n_0$. 
Since the sign stabilizes for $n \geq n_0$, we have
\begin{align*}
    \bigl\|F^{(n+1)}\bigr\|_{\max}^{E_\phi} 
    &\leq \bigl\| \check{E}_\phi F^{(n)}\bigr\|_{\max}^{E_\phi}\\
    & \quad+ \sum_{m=0}^{h-1} \bigl\|\check{E}_h \cdots \check{E}_{h-m} \cdot \epsilon_{h-m-1} \Delta_{k_{h-m-1}} \cdot E_{h-m-2}\cdots E_{-1} C^{(n)}_{-\bs}\bigr\|_{\max}^{E_\phi} \\
    &\leq \rho(\check{E}_\phi)\bigl\|F^{(n)}\bigr\|_{\max}^{E_\phi} + \const \bigl\|C^{(n)}_{-\bs}\bigr\|_{\max}^{E_\phi} \qquad \qquad \qquad \mbox{(\cref{l:spec_radius})} \\
    &\leq s_n\rho(\check{E}_\phi)R_\phi^n + \const\rho(E_\phi)^{n-n_0}\bigl\|C^{(n_0)}_{-\bs}\bigr\|_{\max}^{E_\phi} \\
    &=s_n \rho(\check{E}_\phi)R_\phi^n + \constt\rho(E_\phi)^{n+1} \\
    &\leq (s_n+\constt)R_\phi^{n+1}= s_{n+1}R_\phi^{n+1}.
\end{align*}
Thus the assertion is proved. 
\end{proof}
Combining \cref{l:fromabove_1,l:fromabove_2}, we get an estimate of the algebraic entropies $\cE_\phi^x$ and $\cE_\phi^a$ from above:

\begin{prop}\label{p:estimate_above}
We have $\cE_\phi^x \leq \log R_\phi$ and $\cE_\phi^a \leq \log R_\phi$.
\end{prop}

\begin{proof}
Note that
\[\limsup_{n \to \infty}\frac{1}{n} \log (A_n + B_n) = \max\left\{\limsup_{n \to \infty} \frac{1}{n} \log A_n,\ \limsup_{n \to \infty} \frac{1}{n} \log B_n \right\}\]
for sequences $(A_n),(B_n)$ of positive numbers. From \cref{l:fromabove_1}, we get
\begin{align*}
    \cE_\phi^x &= \limsup_{n \to \infty} \frac{1}{n}\log \degr(\phi^x)^n \\
    &\leq \limsup_{n \to \infty}\left( \frac{1}{n}\log (\|C^{(n)}\|_{\max} + \|F^{(n)}\|_{\max})  + \frac{1}{n}\log K' \right) \\
    &= \max \left\{\limsup_{n \to \infty} \frac{1}{n}\log \|C^{(n)}\|_{\max},\  \limsup_{n \to \infty}\frac{1}{n}\log \|F^{(n)}\|_{\max}  \right\}.
\end{align*}
The first term gives $\mathsf{\Lambda}_{E_\phi}(C^{(n_0)}) = \log \rho(E_\phi)$ by the proof of \cref{p:estimate_below}. On the second term, we replace the uniform norm with the rotated norm $\big\|\cdot\big\|_{\max}^{E_\phi}$ and compute
\begin{align*}
    \limsup_{n \to \infty}\frac{1}{n}\log \|F^{(n)}\|_{\max} 
    &=\limsup_{n \to \infty}\frac{1}{n}\log \big\|F^{(n)}\big\|_{\max}^{E_\phi} \\
    &\leq \limsup_{n \to \infty}\left(\frac{1}{n}\log R_\phi^n + \frac{1}{n}\log s_n \right) \\
    &=R_\phi.
\end{align*}
Here we used \cref{l:fromabove_2} and the fact that the sequence $(s_n)$ is an arithmetic sequence. The proof of the second statement follows from the same line of arguments.
\end{proof}
Combining \cref{p:estimate_below,p:estimate_above}, we get a proof of \cref{t:entropy}.



\subsection{Two-sided sign stability}
Let us consider the case where a path $\gamma : t_0 \to t$ which represents a mutation loop $\phi = [\gamma]_\bs$ is not sign-stable on $\cL_{(t_0)}^{\mathrm{can}} = \interior \cC^+_{(t_0)} \cup \interior \cC^-_{(t_0)}$ but is sign-stable on each of $\interior \cC^+_{(t_0)}$ and $\interior \cC^-_{(t_0)}$.
When it satisfies a suitable condition, we can still compute the algebraic entropy of the cluster transformations induced by such a  mutation loop.

\begin{defi}\label{def:two-sided sign-stab}
Let $\gamma: t_0 \to t$ be a path which represents a mutation loop $\phi = [\gamma]_{\bs}$.
The path $\gamma$ is said to be \emph{two-sided sign-stable} if it is sign-stable on each of $\interior \cC^+_{(t_0)}$ and $\interior \cC^-_{(t_0)}$ with stable signs $\boldsymbol{\epsilon}^+_\gamma$ and $\boldsymbol{\epsilon}^-_\gamma$ respectively, and satisfies the following conditions:
\begin{itemize}
    \item $\boldsymbol{\epsilon}^+_\gamma = - \boldsymbol{\epsilon}^-_\gamma$, and 
    \item the spectral radii of $E^{\boldsymbol{\epsilon}^+_\gamma}_\gamma$ and $E^{\boldsymbol{\epsilon}^-_\gamma}_\gamma$ are the same.
\end{itemize}

\end{defi}

When $\gamma$ is two-sided sign-stable, we denote the spectral radii $\lambda_{\phi, \interior\cC^+_{(t_0)}} = \lambda_{\phi, \interior\cC^-_{(t_0)}}$ by $\lambda_\phi^{(t_0)}$ and still call it the  \emph{cluster stretch factor} of the mutation loop $\phi$.

We will give an example of a two-sided sign-stable mutation loop in \cref{ex:two-sided sign-stab}.

\begin{cor}\label{cor:two-sided}
Let $\phi=[\gamma]_{\bs}$ be a mutation loop with a representation path $\gamma$ which is two-sided sign-stable. Assuming that \cref{p:spec_same} holds true, we have
\begin{align*}
    \cE_\phi^a=\cE_\phi^x = \log \lambda_\phi^{(t_0)}.
\end{align*}
\end{cor}

\begin{proof}
Observe that all the statements in \cref{sec:from below,sec:from above} still hold, except for \cref{l:F-matrix_mutation_stable}. Indeed, they only depends on the sign stability on each of the cones $\interior \cC^+_{(t_0)}$ and $\interior \cC^-_{(t_0)}$. 
The replacement of \cref{l:F-matrix_mutation_stable} is rather simpler than the original one:
\begin{lem}\label{l:two-sided}
Suppose that the representation path $\gamma: t_0 \to t$ of $\phi$ is two-sided sign-stable. Then there exists an integer $n_0 \geq 0$ such that 
\begin{align*}
    F^{(n+1)} = \check{E}_\phi^{(t_0)}F^{(n)}
\end{align*}
for all $n \geq n_0$.
\end{lem}
\begin{proof}
By definition of the two-sided sign stability,
the stable sign $\boldsymbol{\epsilon}^+_\gamma$ on $\interior \cC^+_{(t_0)}$ is the minus of the stable sign $\boldsymbol{\epsilon}^-_\gamma$ on $\interior \cC^-_{(t_0)}$:
$\boldsymbol{\epsilon}^+_\gamma = - \boldsymbol{\epsilon}^-_\gamma$. The former sign is the sequence of tropical signs by \cref{cor:C_as_pres_mat}, and the latter gives the signs of row vectors of the matrices $C^{-\bs;t_0}_{t}$ by \cref{cor:pres_mat_x_negative}. 
Hence the second term in the right-hand side of \cref{t:FuGy} vanishes at vertex $t_m$ along the path $\gamma$, for sufficiently large integer $m$.
\end{proof}
Now \cref{cor:two-sided} can be proved by following the same line as the proof of \cref{t:entropy}, where the statements derived from \cref{l:F-matrix_mutation_stable} are replaced with simpler ones corresponding to \cref{l:two-sided}.
\end{proof}

\section{Check methods and examples}\label{sec:example}
In this section we give several methods for checking sign stability and demonstrate them in concrete examples.
We denote by $\bS \cX_{(t_0)}(\bR^\trop)$ the quotient of  $\cX_{(t_0)}(\bR^\trop)$ by the rescaling $\bR_{>0}$-action, and sometimes we identify it with the sphere $\{x_1^2 + \cdots + x_N^2 = 1\} \subset \cX_{(t_0)}(\bR^\trop)$. When we deal with concrete examples, it is useful to represent a skew-symmetric matrix $B=(b_{ij})_{i,j \in I}$ by a quiver $Q$. It has vertices parametrized by the set $I$ and $|b_{ij}|$ arrows from $i$ to $j$ (resp. $j$ to $i$) if $b_{ij} >0$ (resp. $b_{ji} >0$). Note that that the quiver $Q$ has no loops and $2$-cycles, and the matrix $B$ can be reconstructed from such a quiver.

\subsection{An inductive check method.}
Here we give a method for checking sign stability, assuming one uses a computer. 
First of all, we fix a mutation loop $\phi = [\gamma]_\bs$ with $h(\gamma) = h$.
\begin{enumerate}
\item For $n \geq 1$, examine the following inductive process:
\begin{description}
    \item[\rm(A${}_n$)]
    Decompose $\cX_{(t_0)}(\bR^\trop)$ into the cones $\cC_{\gamma^n}^{\boldsymbol{\epsilon}}$ for $\boldsymbol{\epsilon} \in \{+,-\}^{nh}$. 
    \item[\rm(B${}_n$)]
    For each cone $\cC_{\gamma^n}^{\boldsymbol{\epsilon}}$ such that $\dim \cC_{\gamma^n}^{\boldsymbol{\epsilon}} = N$, check whether it satisfies $\phi^n(\cC_{\gamma^n}^{\boldsymbol{\epsilon}}) \subset \mathrm{int}\, \cC_{\gamma^n}^{\boldsymbol{\epsilon}}$.
    If such a cone is found, then the process terminates. Note that  in this case, the sign sequence $\boldsymbol{\epsilon}$ has the form $(\boldsymbol{\epsilon}_0,\dots,\boldsymbol{\epsilon}_0) \in \{+,-\}^{hn}$. 
    Otherwise, proceed to the step (A${}_{n+1}$). 
\end{description}
If this process terminates at $n=n_0$, then proceed to the next process (2). (If it does not terminate, then this method is not effective for checking if $\phi$ is sign-stable or not.)
\item Chase the orbits of the points $l_{(t_0)}^\pm \in \cL_{(t_0)}^{\mathrm{can}}$ under the action of $\phi$.
If each of them goes to the interior of a common cone $\cC_{\gamma^n}^{\boldsymbol{\epsilon}}$ among those found in (1), then this mutation loop is sign-stable with the stable sign $\boldsymbol{\epsilon}_0 \in \{+,-\}^h$, where $\boldsymbol{\epsilon} = (\boldsymbol{\epsilon}_0,\dots,\boldsymbol{\epsilon}_0) \in \{+,-\}^{hn_0}$ by \cref{prop:sign stab ray}.
\end{enumerate}

Obviously this method can misses a sign-stable mutation loop, but it detects many examples.
We demonstrate these steps below:

\begin{ex}[Markov quiver]\label{ex:Markov}
Here, we demonstrate the inductive check method for a concrete example.
Let $I = \{1,2,3\}$ and $\bs: t \mapsto (N^{(t)}, B^{(t)})$ be a seed pattern such that
\[B^{(t_0)} = 
\begin{pmatrix}
0 & 2 & -2\\
-2 & 0 & 2\\
2 & -2 & 0
\end{pmatrix}
\]
for a vertex $t_0 \in \bT_I$.
The quiver corresponding to this matrix is called Markov quiver \cref{fig:Markov_quiv}.

\begin{figure}[h]
    \centering
    \begin{tikzpicture}[>=latex, scale=0.7]
    \node [draw, circle, inner sep=2] (v1) at (-0.5,1.5) {};
    \node [draw, circle, inner sep=2] (v2) at (-2.5,-1.7) {};
    \node [draw, circle, inner sep=2] (v3) at (1.5,-1.7) {};
    \node at (-0.5,2) {1};
    \node at (-3,-2) {2};
    \node at (2,-2) {3};
    \draw [->>, thick] (v1) edge (v2);
    \draw [->>, thick] (v2) edge (v3);
    \draw [->>, thick] (v3) edge (v1);
    \end{tikzpicture}
    \caption{Markov quiver.}
    \label{fig:Markov_quiv}
\end{figure}
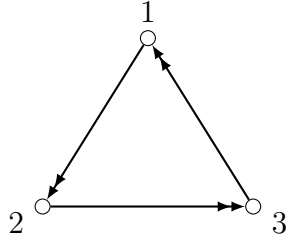

Let us consider the horizontal mutation loop $\phi$ represented by an edge path
\[\gamma: t_0 \overbar{1} t_1 \overbar{2} t_2 \overbar{3} t_3 \overbar{1} t_4 \overbar{2} t_5 \overbar{3} t_6, \]
which is fully-mutating. 
\begin{figure}[h]
    \centering
    \begin{overpic}[width=11cm]{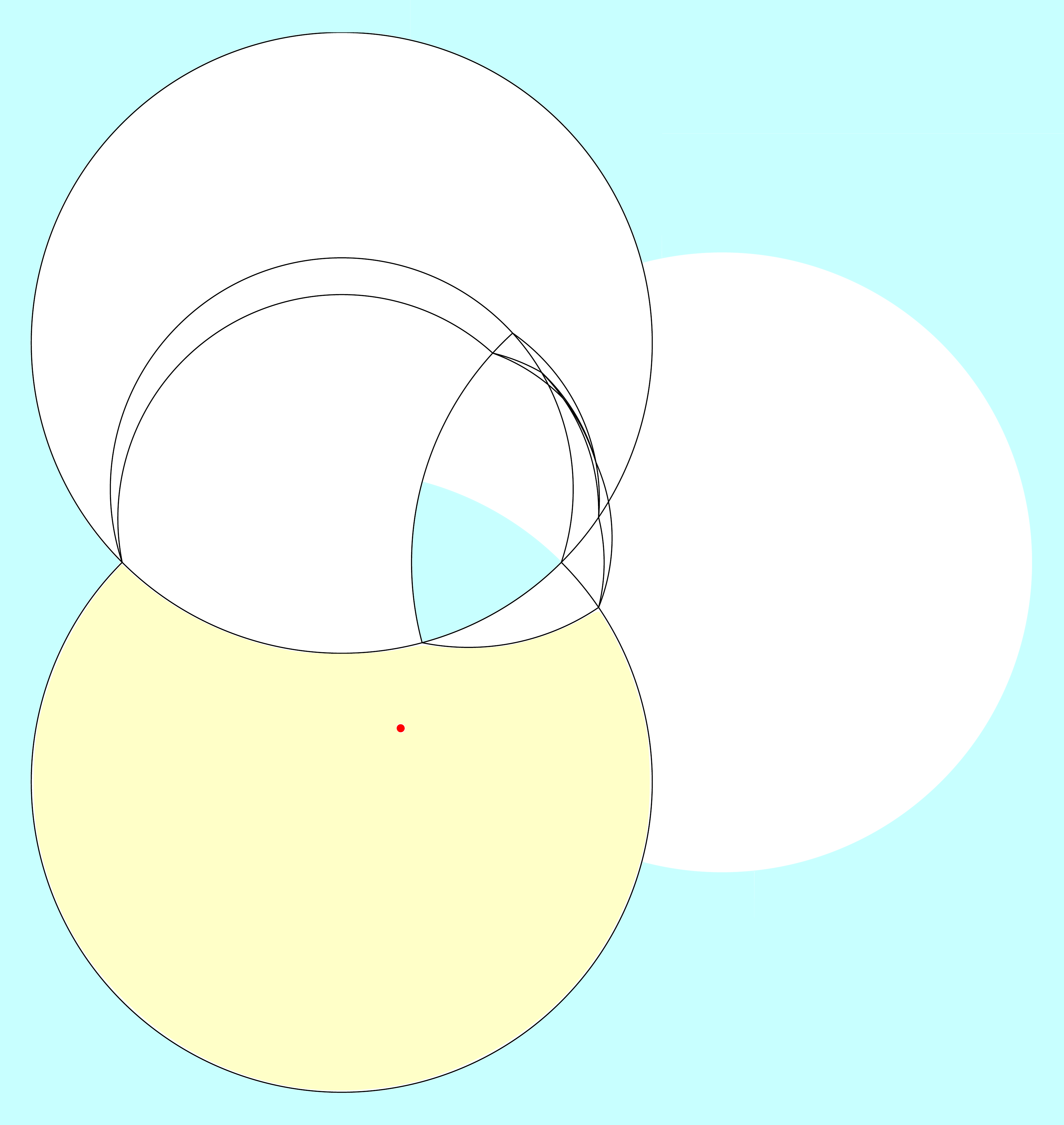}
    \put(130,163){\textcolor{blue}{$\cC^-_{(t_0)}$}}
    \put(250,50){\textcolor{blue}{$\cC^+_{(t_0)}$}}
    \put(55,55){$\cC^{(+,-,+,-,+,-)}_\gamma$}
    \put(125,115){\textcolor{red}{$p_+$}}
    \end{overpic}
    \caption{Cones in $\cX_{(t_0)}(\bR^\trop)$ associated to signs and $\cC^\pm_{(t_0)}$.}
    \label{fig:cones_Markov}
\end{figure}
Then $\cX_{(t_0)}(\bR^\trop)$ is decomposed into 17 full-dimensional cones.
\cref{fig:cones_Markov} shows the image of the intersection of $\bS \cX_{(t_0)}(\bR^\trop)$ and the following objects under the stereographic projection from $[(1,1,1)] \in \bS \cX_{(t_0)}(\bR^\trop)$:
\begin{itemize}
    \item boundary of the cones associated to sign sequences (black lines),
    \item the fixed point
    \[ p_+ = \left[\left(1, -\frac{1+\sqrt{5}}{2}, \frac{-1+\sqrt{5}}{2} \right)\right] \in \cX_{(t_0)}(\bR^\trop) \]
    with the stretch factor $\lambda_\phi = 9+4\sqrt{5}$ (red point),
    \item the cones $\cC^-_{(t_0)}$ and $\cC^+_{(t_0)}$ (blue regions),
    \item the cone associated to the sign sequence $(+,-,+,-,+,-)$ (yellow region).
\end{itemize}
The cone $\cC_\gamma^{(+,-,+,-,+,-)}$ satisfies the condition in (B${}_n$) with $n=1$ and the conditions in (2).
Therefore, we conclude that $\phi$ is sign-stable with stable sign $(+,-,+,-,+,-)$.
Its stable presentation matrices $E_\phi^{(t_0)}$ and $\check{E}_\phi^{(t_0)}$ are given by
\[
E_\phi^{(t_0)} =
\begin{pmatrix}
  9&   6&   4\\
-12&  -7&  -4\\
  4&   2&   1
\end{pmatrix},\quad
\check{E}_\phi^{(t_0)}=
\begin{pmatrix}
 -7&  -4&  12\\
-12&  -7&  20\\
-20& -12&  33
\end{pmatrix}.
\]
Their characteristic polynomials are the same:
\[(\nu - 1) \cdot (\nu^2 - 18\nu + 1).\]
Hence, \cref{p:spec_same} holds true for this example, so the algebraic entropies of the cluster transformations induced by $\phi$ are
\[ \cE^a_\phi = \cE^x_\phi = \log(9+4\sqrt{5}) = 2.88727095035762\dots \]
\end{ex}

\begin{rmk}
In fact, the matrix $B^{(t_0)}$ in \cref{ex:Markov} arises from a triangulation of a once punctured torus.
Moreover, the horizontal mutation loop $\phi$ in \cref{ex:Markov} is actually the third power of the mutation loop $\phi_{LR}$ corresponding to the mapping class called the ``$LR$-transformation'' which gives the smallest stretch factor among the mapping classes on a once punctured torus.
This mutation loop $\phi_{LR}$ is represented by two mutations $(1,2)$ and a permutation $1 \mapsto 2 \mapsto 3 \mapsto 1$.
We can regard the stable sign sequence $(+,-,+,-,+,-)$ as  $((+,-),(+,-),(+,-))$, where each sign sequence $(+,-)$ is the stable sign sequence for the mutation loop $\phi_{LR}$.
\end{rmk}

\subsection{A heuristic check method.}
Fix a mutation loop $\phi$ with $h(\gamma) = h$.
As we mentioned above, the process (1) of the inductive check method may not terminate in finitely many steps.
As an alternative method, when the mutation loop $\phi$ is simple enough, it is well worth trying to find the cone satisfying the conditions in (B${}_n$) with $n=1$ by inspection.
Here is our second method:
\begin{enumerate}
    \item Decompose $\cX_{(t_0)}(\bR^\trop)$ into the cones $\cC^{\boldsymbol{\epsilon}}_\gamma$ for $\boldsymbol{\epsilon} \in \{+,-\}^h$. 
    \item Find a cone $\cC \subset \cX_{(t_0)}(\bR^\trop)$ such that
    \begin{itemize}
        \item $\dim \cC = N$,
        \item $\cC \subset \cC^{\boldsymbol{\epsilon}_0}_\gamma$ for a sign sequence $\boldsymbol{\epsilon}_0 \in \{+,-\}^h$, 
        \item $\phi(\cC) \subset \cC$ and $\phi^{n_0}(\cC \setminus \{0\}) \subset \mathrm{int}\, \cC$ for some $n_0 \geq 1$.
    \end{itemize}
    \item Chase the orbit of the points $l_{(t_0)}^\pm \in \cL_{(t_0)}^{\mathrm{can}}$ under the action of $\phi$.
    If each of them goes to the interior of the cone $\cC$ found in (2), then this mutation loop is sign-stable with the stable sign $\boldsymbol{\epsilon}_0$ by \cref{prop:sign stab ray}.
\end{enumerate}
We will refer to the cone $\cC$ as an \emph{invariant cone} of $\phi$.
Clearly, the choice of an invariant cone is not unique: for example, we don't need to take the maximal one.

\begin{ex}[Kronecker quiver]\label{ex:kronecker}
Let $I:=\{1,2\}$, and $\bs: t \mapsto (N^{(t)}, B^{(t)})$ be a seed pattern such that
\[ B^{(t_0)} = \begin{pmatrix}0 & -\ell \\ \ell & 0\end{pmatrix} \]
for a vertex $t_0 \in \bT_I$ and an integer $\ell \geq 2$.
The quiver corresponding to this matrix is called Kronecker quiver \cref{fig:kronecker_quiv}.
\begin{figure}[h]
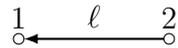

    \centering
    \begin{align*}
    \tikz[>=latex]{\draw(0,0) circle(2pt) node[above]{$1$}; \draw(2,0) circle(2pt) node[above]{$2$}; \qarrow{2,0}{0,0} \node at (1,0.3){$\ell$};}
    \end{align*}
    \caption{Kronecker quiver.}
    \label{fig:kronecker_quiv}
\end{figure}
Let us consider the horizontal mutation loop $\phi$ represented by an edge path
\[\gamma: t_0 \overbar{1} t_1 \overbar{2} t_2, \]
which is fully-mutating. 
The tropical cluster $\cX$-variety $\cX_{(t_0)}(\bR^\trop)$ is decomposed into the following four domains of linearity $\cC^{\boldsymbol{\epsilon}}_\gamma$ for $\phi$ with the presentation matrix $E^{\boldsymbol{\epsilon}}_\gamma$ for $\boldsymbol{\epsilon} \in \{+, -\}^2$:
\begin{table}[h]
    \centering
    \caption{Domains of linearity.}
    \vspace{-5pt}
    \begin{tabular}{c||c|c}
        \hline
        $\boldsymbol{\epsilon}$ & $\cC^{\boldsymbol{\epsilon}}_\gamma$ & $E^{\boldsymbol{\epsilon}}_\gamma$ \\
        \hline \hline
        $(+,+)$ & $\{x_1 \geq 0,\ \ell x_1 + x_2 \geq 0\}$ & $\begin{pmatrix}\ell^2-1 & \ell \\ -\ell & -1\end{pmatrix}$\\
        \hline
        $(+,-)$ & $\{x_1 \geq 0,\ \ell x_1 + x_2 \leq 0\}$ & $\begin{pmatrix}-1 & 0 \\ -\ell & -1\end{pmatrix}$\\
        \hline
        $(-,+)$ & $\{x_1 \leq 0,\ x_2 \geq 0\}$ & $\begin{pmatrix}-1 & \ell \\ 0 & -1\end{pmatrix}$\\
        \hline
        $(-,-)$ & $\{x_1 \leq 0,\ x_2 \leq 0\}$ & $\begin{pmatrix}-1 & 0 \\ 0 & -1\end{pmatrix}$
    \end{tabular}
    \label{tab:dom_kronecker_l}
\end{table}

\begin{figure}[h]
    \centering
    \begin{tikzpicture}[scale=1.2]
    \coordinate (v1) at (0,0) {};
    \coordinate (v3) at (0,3) {};
    \coordinate (v2) at (-3,0) {};
    \coordinate (v4) at (0,-3) {};
    \coordinate (v5) at (1.5,-3) {};
    \coordinate (v6) at (3,-3) {};
    \filldraw [draw=yellow!35, fill=yellow!35] (v3) -- (v1) -- (v6) -- (3,3) -- (v3);
    \draw  (v1) edge (v2);
    \draw  (v1) edge (v3);
    \draw  (v1) edge (v4);
    \draw  (v1) edge (v5);
    \node at (-1.5,1.5) {$\cC^{(-,+)}_\gamma$};
    \node at (-1.5,-1.5) {$\cC^{(-,-)}_\gamma$};
    \node at (0.6,-2.5) {$\cC^{(+,-)}_\gamma$};
    \node at (1.75,0.15) {$\cC^{(+,+)}_\gamma$};
    \node[color=orange] at (1.7,1.3) {$\cC$};
    \coordinate (v7) at (2,-3) {} {};
    \coordinate (v8) at (3,-2) {} {};
    \draw[red, thick]  (v1) edge (v7);
    \draw[red, thick]  (v1) edge (v8);
    \node[red] at (2,-2.5) {$p_-$};
    \node[red] at (2.6,-1.3) {$p_+$};
    \end{tikzpicture}
    \caption{Domains of linearity and an invariant cone (yellow region).}
    \label{fig:cones_kroc}
\end{figure}
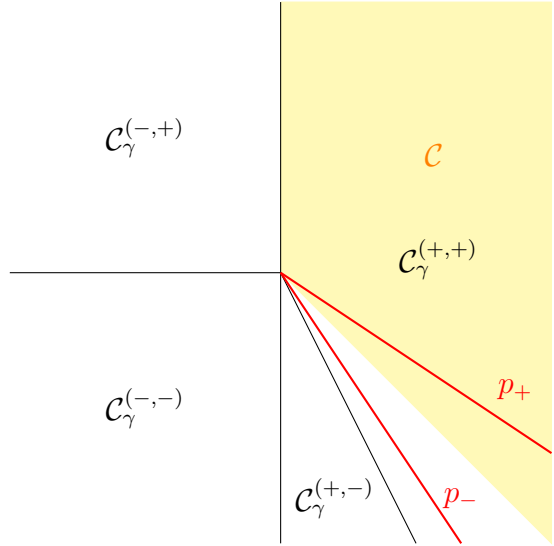
See \cref{fig:cones_kroc}. 
In this case, the process (1) in the inductive check method does not terminate.
However, we can take an invariant cone $\cC$ as 
\[ \cC = \{ x_1 \geq 0,\ x_1 + x_2 \geq 0 \} \subset \cC^{(+,+)}_\gamma. \]
Indeed, for each point $(x_1,x_2) \in \cC$, putting $(x'_1, x'_2) := \phi(x_1,x_2)$ we have
\begin{align*}
    x'_1 &= (\ell^2-1)x_1 + \ell x_2 \geq (\ell^2 - \ell - 1)x_1 \geq 0,\\
    x'_1 + x'_2 &= (\ell^2 - \ell -1) x_1 + (\ell - 1) x_2 \geq (\ell^2 -2 \ell) x_1 \geq 0.
\end{align*}
Hence we have $\phi(\cC) \subset \cC$. One can also see from the above computation that $\phi(\cC \setminus \{0\}) \subset \mathrm{int}\, \cC$. 
Furthermore, by a direct calculation, one can check that
\[\phi(\cC_\gamma^{(+,-)}) = \cC^{(-,+)}_\gamma,\ \phi(\cC_\gamma^{(-,+)} \cup \cC_\gamma^{(-,-)}) \subset \cC.\]
Hence we can conclude that $\phi$ is sign stable.
The stable presentation matrices $E_\phi$ and $\check{E}_\phi$ have the same characteristic polynomial
\[ \nu^2 + (-\ell^2 + 2)\nu +1. \]
Hence \cref{p:spec_same} holds true in this case, so the algebraic entropies of the cluster transformations induced by $\phi$ are obtained as
\[ \cE^x_\phi = \cE^a_\phi = \log\left( \frac{\ell^2 - 2 + \ell \sqrt{\ell^2 - 2} }{2} \right). \]
\end{ex}

\begin{rmk}
The mutation loop in \cref{ex:kronecker} is so simple that we can describe the entire dynamics on $\cX_{(t_0)}(\bR^\trop)$.
First, the action of $\phi$ has only two fixed points $p_+, p_-$ in $\bS \cX_{(t_0)}(\bR^\trop)$ with stretch factors $\lambda_+, \lambda_-$, respectively:
\[ p_\pm = \left[\left(1, \frac{-\ell \pm\sqrt{\ell^2 -4}}{2} \right)\right] ,\  \lambda_\pm = \frac{\ell^2 - 2 \pm \ell \sqrt{\ell^2 - 2} }{2}. \]
Moreover we have $p_\pm \in \bS \cC^{(+,+)}$ (see \cref{fig:cones_kroc}).
Let we consider the following two cones:
\begin{align*}
    \cC^{(+,+)}_+ &:= \left\{ x_1 \geq 0,\ x_1 + \frac{-\ell -\sqrt{\ell^2 -4}}{2} x_2 >0 \right\},\\
    \cC^{(+,+)}_- &:= \left\{ \ell x_1 + x_2 \geq 0,\ x_1 + \frac{-\ell -\sqrt{\ell^2 -4}}{2} x_2 <0 \right\}.
\end{align*}
so that $\cC^{(+,+)}_\gamma \setminus p_- = \cC^{(+,+)}_+ \sqcup \cC^{(+,+)}_- $. 
By a direct calculation, one can see that each point in $\bS \cC^{(+,+)}_+$ converges to $p_+$ and that each point in $\bS \cC^{(+,+)}_-$ leaves this region and travels $\bS\cC^{(+,-)}_\gamma, \bS\cC^{(-,+)}_\gamma$, and $\bS\cC^{(-,-)}_\gamma$ in this order, and finally converges to $p_+$. Namely, the action of $\phi$ has the \emph{north-south dynamics} on $\bS \X_{(t_0)}(\bR^\trop)$ with the attracting (resp. repelling) point $p_+$ (resp. $p_-$). 
Therefore $\cC^{(+,+)}_+$ is the maximal invariant cone.
\end{rmk}

\subsection{Mutation loops of length one}
Here we give a family of examples for which the heuristic check method can be effectively applied. 
Let $I:=\{1,\dots,N\}$, and consider a seed pattern $\bs: t \mapsto (N^{(t)}, B^{(t)})$. Fix a vertex $t_0 \in \bT_I$ and put $B:=B^{(t_0)}$. Let us consider an edge path $\gamma: t_0 \overbar{1} t_1$ and a permutation $\sigma$ on $I$ which is given by $\sigma(i):=i-1$ mod $N$ for $i=1,\dots,N$. The following classification theorem is due to Fordy--Marsh:

\begin{thm}[Fordy--Marsh \cite{FM11}]\label{thm:FM}
The condition
\begin{align}\label{eq:period one}
    \sigma.B^{(t_1)} = B
\end{align}
holds if and only if there exists an integer vector $\mathbf{a}=(a_1,\dots,a_{N-1})$ such that $a_j = a_{N-j}$ for $j=1,\dots,N-1$ and the skew-symmetric matrix $B=(b_{ij})_{i,j \in I}$ satisfies the following conditions:
\begin{align}
    b_{i,j+1}=a_j, \quad b_{i+1,j+1} = b_{ij} + a_i [-a_j]_+ -a_j [-a_i]_+ \label{eq:FM_cond}
\end{align}
for all $i,j \in \{1,\dots,N-1\}$. 
\end{thm}
The condition \eqref{eq:period one} can be paraphrased that the mutation sequence given by the edge path $\gamma$ followed by the permutation $\sigma$ defines a mutation loop $\phi$. Although it is not a horizontal mutation loop and slightly sticks out our scope in this paper, the $N$-th power $\phi^N$ is a horizontal mutation loop represented by the fully-mutating edge path 
\begin{align*}
    \gamma^N:t_0 \overbar{1} t_1 \overbar{2} t_2 \overbar{3}\dots \overbar{N} t_N.
\end{align*}
Sign stability can be naturally generalized to this situation as follows. The sign of $\gamma$ at $w \in \X_{(t_0)}(\bR^\trop)$ is defined to be $\epsilon_\gamma(w):=\sgn (x_1^{(t_0)}(w))$. The presentation matrix of $\phi$ at $w$ is given by $E_\gamma^\epsilon:=P_\sigma E_{1,\epsilon}^{(t_0)}$, where $\epsilon:=\epsilon_\gamma(w)$ and $P_\sigma:=(\delta_{i,\sigma(j)})_{i,j \in I}$ is the presentation matrix of $\sigma$. 
For a scaling invariant subset $\cL \subset \X_{(t_0)}(\bR^\trop)$, $\phi$ is said to be sign-stable if there exists a strict sign $\epsilon_\gamma^\stab \in \{+,-\}$ such that for each $w \in \cL\setminus \{0\}$, there exists an integer $n_0 \in \bN$ such that $\epsilon_\gamma(\phi^n(w)) = \epsilon_\gamma^\stab$ for $n \geq n_0$. 
If $\phi$ is sign-stable, then so is $\phi^N$. See \cite{IK20} for a general framework. Since the sign-stability for $\phi$ is easier to check than that for $\phi^N$, we are going to work with $\phi$. 
Explicitly, the presentation matrix $E_\gamma^\epsilon$ is given as follows:
\begin{align*}
    E_\gamma^\epsilon = 
    \begin{pmatrix}
    0&1&& & \\
     &0&1& & \\
     & &\ddots&\ddots& \\
     & & &0&1\\
    1& & & &0 
    \end{pmatrix}
    \begin{pmatrix}
    -1& & & & \\
    [\epsilon a_1]_+&1& & & \\
    [\epsilon a_2]_+& &1& & \\
    \vdots& & &\ddots& \\
    [\epsilon a_{N-1}]_+& & & &1
    \end{pmatrix}
    =
    \begin{pmatrix}
    [\epsilon a_1]_+&1& & & \\
    [\epsilon a_2]_+&0&1& & \\
    \vdots& &0&\ddots& \\
    [\epsilon a_{N-1}]_+& & &\ddots&1\\
    -1& & & &0
    \end{pmatrix}.
\end{align*}
Then we can find an invariant cone as follows:

\begin{lem}\label{l:FMloops_invcone}
\begin{enumerate}
    \item The cone
\begin{align*}
    \cC_{[+]}:=\{x_1 \geq 0, ~[a_{i-1}]_+x_1 + x_i \geq 0 \mbox{ for $i=2,\dots,N$}\}
\end{align*}
satisfies $\phi(\cC_{[+]}) \subset \cC_{[+]}$ if and only if $a_1 \geq 2$. Moreover in this case, we have $\phi^N(\cC_{[+]}\setminus\{0\}) \subset \mathrm{int}\,\cC_{[+]}$. 
\item The cone
\begin{align*}
    \cC_{[-]}:=\{x_1 \leq 0, ~[-a_{i-1}]_+x_1 + x_i \leq 0 \mbox{ for $i=2,\dots,N$}\}
\end{align*}
satisfies $\phi(\cC_{[-]}) \subset \cC_{[-]}$ if and only if $a_1 \leq -2$. Moreover in this case, we have $\phi^N(\cC_{[-]}\setminus\{0\}) \subset \mathrm{int}\,\cC_{[-]}$. 
\end{enumerate}
\end{lem}

\begin{proof}
First suppose $a_1 \geq 2$. For a point $w \in \cC_{[+]}$, let $w':=\phi(w)$. Let $\mathbf{x}=(x_1,\dots,x_N)$, $\mathbf{x}'=(x'_1,\dots,x'_N)$ be the coordinate vectors defined by $x_i:=x_i^{(t_0)}(w)$ and $x'_i:=x_i^{(t_0)}(w')$ for $i=1,\dots,N$. Then from $(\mathbf{x}')^T = E_\gamma^+ \mathbf{x}^T$, we get
\begin{align*}
    &x'_1 = a_1 x_1 + x_2 \geq 0, \\
    &[a_{i-1}]_+x'_1 + x'_i = [a_{i-1}]_+(a_1 x_1 + x_2) + ([a_i]_+ x_1 + x_{i+1}) \geq 0, \\
    &[a_{N-1}]_+x'_1 + x'_N =a_1(a_1 x_1 + x_2) -x_1 
    = (a_1^2-1)x_1 + x_2 > a_1 x_1 + x_2 \geq 0.
\end{align*}
Here $i=2,\dots,N-1$. Thus we have $\phi(\cC_{[+]}) \subset \cC_{[+]}$. Furthermore, let $\Phi^+_1:=\cC_{[+]} \cap \{x_1=0\}$ and $\Phi^+_i:=\cC_{[+]} \cap \{[a_{i-1}]_+x_1+x_i=0\}$ for $i=2,\dots,N$ be facets of the cone $\cC_{[+]}$. Then the above computation shows that $\phi(\cC_{[+]}) \subset \cC_{[+]} \setminus \Phi_N^+$,
\[
\begin{tikzcd}
    \cC_{[+]} \setminus \Phi_N^+ \ar[r, "E_\gamma^+"', "\phi"] & \cC_{[+]} \setminus \Phi_{N-1}^+ \ar[r, "E_\gamma^+"', "\phi"] &
    \cdots \ar[r, "E_\gamma^+"', "\phi"] &  \cC_{[+]} \setminus \Phi_2^+,
\end{tikzcd}
\]
and $\phi(\cC_{[+]} \setminus \Phi_2^+) \subset \mathrm{int}\,\cC_{[+]}$.
Thus we have $\phi^N(\cC_{[+]}) \subset \mathrm{int}\,\cC_{[+]}$. 
On the other hand, when $a_1 \leq 1$, $\mathbf{x}=(2,-1,0,\dots,0) \in \cC_{[+]}$ satisfies $\mathbf{x}' \notin \cC_{[+]}$. 
The second assertion can be similarly proved. 
\end{proof}

\begin{cor}\label{c:FMloops_stability}
Let $\epsilon \in \{+,-\}$ be a sign. 
If $\epsilon a_1 \geq 2$ and $\epsilon a_i \geq 0$ for $i=2,\dots,N-1$ holds, then $\phi$ is sign-stable on the set $\cL_{(t_0)}^{\mathrm{can}}$ with the stable sign $\epsilon$. Moreover \cref{p:spec_same} holds true in this case. 
The cluster stretch factor is the positive solution of the equation $\nu^N - \sum_{i=1}^{N-1}\epsilon a_i \nu^{N-i} +1 = 0$ which is largest among the solutions in absolute value.  
\end{cor}

\begin{proof}
It is enough to consider the case $\epsilon=+$. We need to check that the cones $\cC^+_{(t_0)}$ and $\cC_{(t_0)}^-$ are send into the cone $\cC_{[+]}$ by an iterated action of $\phi$. It is clear from the definition of $\cC_{[+]}$ that we already have $\cC_{(t_0)}^+ \subset \cC_{[+]}$. For a point $w_- \in \cC_{(t_0)}^-$, let $x_i:=x_i^{(t_0)}(w_-) \leq 0$ be its coordinates for $i=1,\dots,N$. Since the action of $\phi$ is presented by the matrix 
\begin{align*}
    E_\gamma^-= \begin{pmatrix}
    0&1& & & \\
    0&0&1& & \\
    \vdots& &0&\ddots& \\
    0& & &\ddots&1\\
    -1& & & &0
    \end{pmatrix}
\end{align*}
whenever the first coordinate of a point is non-positive, we can compute the $\phi$-orbit of $w_-$ as 
\begin{align*}
    \begin{pmatrix}x_1\\x_2\\ \vdots\\x_{N-1}\\x_N \end{pmatrix}
    \xmapsto{E_\gamma^-}
    \begin{pmatrix}x_2\\x_3\\ \vdots\\x_N\\-x_1 \end{pmatrix}
    \xmapsto{E_\gamma^-}
    \begin{pmatrix}x_3\\x_4\\ \vdots\\-x_1\\-x_2 \end{pmatrix}
    \xmapsto{E_\gamma^-}\dots \xmapsto{E_\gamma^-}
    \begin{pmatrix}-x_1\\-x_2\\ \vdots\\-x_{N-1}\\-x_N \end{pmatrix}.
\end{align*}
Hence $\phi^N(\cC_{(t_0)}^-) = \cC_{(t_0)}^+ \subset \cC_{[+]}$. Combining with \cref{l:FMloops_invcone}, we get $\phi^{2N}(\cL_{(t_0)}^{\mathrm{can}}\setminus \{0\}) \subset \mathrm{int}\,\cC_{[+]}$. 
Thus $\phi$ is sign-stable on the set $\cL_{(t_0)}^{\mathrm{can}}$ with the stable sign $(+)$. 

The characteristic polynomial of $E_\gamma^+$ is given by $P_\gamma^+(\nu):=\nu^N - \sum_{i=1}^{N-1}a_i \nu^{N-i} +1$, which is palindromic. Hence \cref{p:spec_same} holds true in this case. The case $\epsilon=-$ is similarly proved.
\end{proof}
\cref{c:FMloops_stability} gives a partial confirmation of \cite[Conjecture 3.1]{FH14}. 

\begin{rem}
The mutation loop considered in \cref{ex:kronecker} is obtained as the square of the mutation loop for the vector $\mathbf{a}=(-\ell)$. In the case $\ell=1$, we get the seed pattern of type $A_2$ and hence $\phi$ is periodic. In particular it has no invariant cone other than $\{0\}$, and not sign-stable. This example indicates a reason why we need the condition $|a_1| \geq 2$ in \cref{l:FMloops_invcone,c:FMloops_stability}.
\end{rem}

\begin{rem}\label{rem:FMloop_convexity}
Since the path $\gamma^N$ is fully-mutating, by \cref{thm:PF property} and \cref{rem:power_convexity}, an eigenvector corresponding to the cluster stretch factor can be found in the stable cone $\cC_\gamma^\stab$. Hence by \cref{rem:intrinsicality}, the cluster stretch factor is intrinsic to the mutation loop $\phi$.
\end{rem}

Finally, we give an example of two-sided sign-stable mutation loop. See \cref{def:two-sided sign-stab}.

\begin{ex}[{\cite[Example 3.7]{FH14}}]\label{ex:two-sided sign-stab}
Let $I = \{1,2,3,4,5,6\}$ and $\bs: t \mapsto (N^{(t)}, B^{(t)})$ be the seed pattern with the initial exchange matrix
\[
B^{(t_0)} = 
\begin{pmatrix}
0& -2& 2& 4& 2& -2\\
2& 0& -6& -6& 0& 2\\
-2& 6& 0& -6& -6& 4\\
-4& 6& 6& 0& -6& 2\\
-2& 0& 6& 6& 0& -2\\
2& -2& -4& -2& 2& 0
\end{pmatrix}.
\]

\begin{figure}[h]
    \centering
    \begin{tikzpicture}[>=latex, scale=1.2]
    \node [draw, circle, inner sep=2pt] (v2) at (0,3.5) {};
    \node [draw, circle, inner sep=2pt] (v1) at (-1.95,2.5) {};
    \node [draw, circle, inner sep=2pt] (v3) at (-1.95,0) {};
    \node [draw, circle, inner sep=2pt] (v4) at (0,-1) {};
    \node [draw, circle, inner sep=2pt] (v5) at (1.95,0) {};
    \node [draw, circle, inner sep=2pt] (v6) at (1.95,2.5) {};
    \node at (0,4) {1};
    \node at (-2.45,2.85) {2};
    \node at (-2.45,-0.35) {3};
    \node at (0,-1.5) {4};
    \node at (2.45,-0.35) {5};
    \node at (2.45,2.85) {6};
    \draw [thick, ->>] (v1) edge (v2);
    \draw [thick, ->>] (v2) edge (v3);
    \draw [thick, ->>>>] (v2) edge (v4);
    \draw [thick, ->>] (v2) edge (v5);
    \draw [thick, ->>] (v6) edge (v2);
    \draw [thick, ->>>>>>] (v3) edge (v1);
    \draw [thick, ->>>>>>] (v4) edge (v1);
    \draw [thick, ->>] (v1) edge (v6);
    \draw [thick, ->>>>>>] (v4) edge (v3);
    \draw [thick, ->>>>>>] (v5) edge (v3);
    \draw [thick, ->>>>] (v3) edge (v6);
    \draw [thick, ->>>>>>] (v5) edge (v4);
    \draw [thick, ->>] (v4) edge (v6);
    \draw [thick, ->>] (v6) edge (v5);
    \end{tikzpicture}
    \caption{The quiver corresponding to $B^{(t_0)}$ in \cref{ex:two-sided sign-stab}.}
    \label{fig:my_label}
\end{figure}

This seed pattern satisfies the condition \eqref{eq:FM_cond} in \cref{thm:FM}, and hence the edge path $\gamma: t_0 \overbar{1} t_1$ followed by the permutation $\sigma: i \mapsto i-1 \mod 6$ gives a mutation loop $\phi$.
Let $E^\pm_\gamma$ denote the presentation matrix of $\phi$ on the cone $\{ \pm x_1 \geq 0\}$:
\[
E^+_\gamma = 
\begin{pmatrix}
 0&  1&  0&  0&  0&  0\\
 2&  0&  1&  0&  0&  0\\
 4&  0&  0&  1&  0&  0\\
 2&  0&  0&  0&  1&  0\\
 0&  0&  0&  0&  0&  1\\
-1&  0&  0&  0&  0&  0
\end{pmatrix},\quad
E^-_\gamma = 
\begin{pmatrix}
 2&  1&  0&  0&  0&  0\\
 0&  0&  1&  0&  0&  0\\
 0&  0&  0&  1&  0&  0\\
 0&  0&  0&  0&  1&  0\\
 2&  0&  0&  0&  0&  1\\
-1&  0&  0&  0&  0&  0
\end{pmatrix}.
\]
\begin{lem}
Let $\cC_{[+]}$ and $\cC_{[-]}$ be the cones defined by
\begin{align*}
    \cC_{[+]} &:= \Bigg\{
    \begin{array}{c}
         x_1 \geq 0,\quad x_2 \geq 0,\quad 2x_1 + x_3 \geq 0,\quad 4x_1 + x_2 + x_4 \geq 0 , \\
         6 x_1 + 4 x_2 + 2 x_3 + x_5 \geq 0,\quad 16 x_1 + 6 x_2 + 4 x_3 + 2 x_4 + x_6 \geq 0
    \end{array}
    \Bigg\},\\
    \cC_{[-]} &:= \Bigg\{
    \begin{array}{c}
         x_1 \leq 0,\quad x_1 + x_2 \leq 0,\quad 2 x_1 + x_2 + x_3 \leq 0,\quad 4 x_1 + 2 x_2 + x_3 + x_4 \leq 0, \\
         8 x_1 + 4 x_2 + 2 x_3 + x_4 + x_5 \leq 0,\quad 16 x_1 + 8 x_2 + 4 x_3 + 2 x_4 + x_5 + x_6 \leq 0
    \end{array}
    \Bigg\}.
\end{align*}
Then for each $\epsilon \in \{+,-\}$ we have $\phi(\cC_{[\epsilon]}) = E^\epsilon_\gamma(\cC_{[\epsilon]}) \subset \cC_{[\epsilon]}$.
Moreover, for any point $x \in \cC^\epsilon_{(t_0)}$ there exist an integer $n_0 >0$ such that $\phi^n(x) \in \cC_{[\epsilon]}$ for all $n > n_0$.
\end{lem}
\begin{proof}
Let us first consider the case $\epsilon=+$. 
For a point $w \in \cC_{[+]}$, let $w':=\phi(w)$.
Let $\mathbf{x}=(x_1,\dots,x_6)$, $\mathbf{x}'=(x'_1,\dots,x'_6)$ be the coordinate vectors defined by $x_i:=x_i^{(t_0)}(w)$ and $x'_i:=x_i^{(t_0)}(w')$ for $i=1,\dots,6$.
Then,
\begin{align*}
    &x'_1 = x_2 \geq 0,\\
    &x'_2 = 2x_1 + x_3 \geq 0,\\
    &2x'_1 + x'_3 = 4x_1 + x_2 + x_4 \geq 0,\\
    &4x'_1 + x'_2 + x'_4 = 6 x_1 + 4 x_2 + 2 x_3 + x_5 \geq 0 ,\\
    &6 x'_1 + 4 x'_2 + 2 x'_3 + x'_5 = 16 x_1 + 6 x_2 + 4 x_3 + 2 x_4 + x_6 \geq 0,\\
    &16 x'_1 + 6 x'_2 + 4 x'_3 + 2 x'_4 + x'_6 = 31x_1 + 16x_2 + 6x_3 + 4x_4 + 2x_5\\
    & \hspace{5.1cm} = 2(6 x_1 + 4 x_2 + 2 x_3 + x_5) + 2(2x_1 + x_3) + 15x_1 + 8x_2 \geq 0.
\end{align*}
That is, $\phi(\cC_{[+]}) \subset \cC_{[+]}$.
The case $\epsilon=-$ can be proved by a similar computation. 
The second statement can be checked by chasing the orbit of (the minus of) the standard basis vectors of $\cX_{(t_0)}(\bR^\trop)$.
\end{proof}
Thus the mutation loop $\phi$ is sign-stable on each of the cones $\cC^+_{(t_0)}$ and $\cC^-_{(t_0)}$.
The characteristic polynomial of the presentation matrix $E^\epsilon_\gamma$ is given by
\[(\nu^2 +\epsilon \nu + 1) \cdot (\nu^4 - \nu^3 - 2 \nu^2 - \nu + 1)\]
for $\epsilon \in \{+,-\}$, which is palindromic. In particular \cref{p:spec_same} holds true on each cone $\cC^\epsilon_{(t_0)}$.
The spectral radii of the matrices $E^+_\gamma$ and $E^-_\gamma$ are the same, which is the largest solution $\lambda_{\max}=2.0810189966245\dots$ to the equation $\nu^4 - \nu^3 - 2 \nu^2 - \nu + 1=0$.
Therefore the mutation loop $\phi$ is two-sided sign-stable, so the algebraic entropies of the induced cluster transformations are given by
\[ \cE^a_\phi = \cE^x_\phi = \log \lambda_{\max} = 0.73285767597364\dots \]
\end{ex}

\appendix

\section{Terminology from discrete linear dynamical systems}
We recollect here some basic terminology concerning discrete dynamical systems on $\bR^N$ obtained as the iteration of a real $N \times N$-matrix $E$. We refer the reader to \cite{CK} for proofs.

\subsection{Real Jordan normal form of a real matrix}\label{subsec:real Jordan}
First we recall the real Jordan normal form of a real matrix. 
The \emph{real Jordan block} $J(\nu,m)$ of algebraic multiplicity $m$ with eigenvalue $\nu \in \bC$ is defined to be
\begin{align*}
     J(\nu,m):= 
    \begin{pmatrix}
    \nu&1& & \\
     &\ddots&\ddots& \\
     & &\ddots&1 \\
     & & &\nu
    \end{pmatrix}
\end{align*}
if $\nu$ is real, and 
\begin{align*}
     J(\nu,m):= 
    \begin{pmatrix}
    \xi&-\eta&1&0& & & & & &\\
    \eta&\xi&0&1& & & & & &\\
     & &\xi&-\eta& & & & & &\\
     & &\eta&\xi& & & & & &\\
     & & & &\ddots& & & &\\
     & & & & & &\xi&-\eta&1&0\\
     & & & & & &\eta&\xi&0&1\\
     & & & & & & & &\xi&-\eta\\
     & & & & & & & &\eta&\xi
    \end{pmatrix}
\end{align*}
if $\nu$ is non-real. Here we write $\nu=\xi\pm i\eta$ and $\eta>0$. Note that the size of the matrix $J(\nu,m)$ is $2m$ if $\nu$ is non-real. 

Recall that if a complex number $\nu$ is an eigenvalue of a real matrix $E$, then so is its complex conjugate $\overline{\nu}$. Therefore we can pick the one with positive imaginary part.

\begin{thm}[{\cite[Theorem 1.2.3]{CK}}]\label{t:real Jordan}
For any real matrix $E$, let $\nu_1,\dots,\nu_{r}$ be its eigenvalues whose imaginary part is non-negative. 
Then there exists an invertible real matrix $S \in GL_N(\bR)$ such that $S^{-1}ES = J(\nu_1,m_1) \oplus \cdots \oplus J(\nu_r,m_r)$, where each $m_k$ is the algebraic multiplicity of the eigenvalue $\nu_k$. We call the right-hand side the \emph{real Jordan normal form} of $E$. We call the column vectors of $S$ the \emph{real Jordan basis vectors}.
\end{thm}
We have the corresponding \emph{real generalized eigenspace decomposition} 
\begin{align*}
    \bR^N = \widetilde{V}_{\nu_1} \oplus \cdots \oplus \widetilde{V}_{\nu_r}.
\end{align*}
Here 
\begin{align*}
    \widetilde{V}_{\nu_k}:= \{ \Re(v) \mid v \in \mathrm{ker}(E^\bC-\nu_k)^{m_k} \} + \{ \Im(v) \mid v \in \mathrm{ker}(E^\bC-\nu_k)^{m_k} \}
\end{align*}
for an eigenvalue $\nu_k$ of algebraic multiplicity $m_k$, where we let $E$ act on the complex vector space $\bC^N$. Note that an eigenvector of $E$ with real eigenvalue $\nu_k$ is real. Hence $\widetilde{V}_{\nu_k}= \mathrm{ker}(E-\nu_k)^{m_k}$ in this case, which recovers the usual generalized eigenspace.

\subsection{Lyapunov spaces and Lyapunov exponents}\label{subsec:Lyapunov}
Let $E$ be a real $N \times N$-matrix. 
Let $\nu_1,\dots,\nu_{r}$ be the eigenvalues of $E$ whose imaginary part is non-negative. Denote the distinct modulus $|\nu_k|$ of eigenvalues $\nu_k$ by $\lambda_j$, and order them as $\lambda_1 >\dots > \lambda_l$ with $1 \leq l \leq r$. Namely, we have $\{|\nu_1|,\dots,|\nu_r|\}= \{\lambda_1,\dots,\lambda_l\}$. 
The largest modulus is denoted by $\rho(E):=\lambda_1$, and called the \emph{spectral radius} of $E$. 
We define the \emph{Lyapunov space} of $\lambda_j$ to be 
\begin{align*}
    L(\lambda_j) := \bigoplus_{k:~ |\nu_k|=\lambda_j} \widetilde{V}_{\nu_k}.
\end{align*}
Then we have the direct sum decomposition $\bR^N = L(\lambda_1) \oplus \dots \oplus L(\lambda_\ell)$. 
Let $1 \leq p \leq r$ be the largest integer such that $\widetilde{V}_{\lambda_p} \subset L(\lambda_1)$. Namely, $\rho(E)=|\nu_1| = \dots =|\nu_p| > |\nu_{p+1}| \geq \dots \geq |\nu_r|$.

For an invertible real matrix $E \in GL_N(\bR)$ and a point $v \in \bR^N \setminus \{0\}$, the \emph{Lyapunov exponent} of the orbit $\{E^n(v)\}_{n \geq 0}$ is defined to be
\begin{align*}
    \mathsf{\Lambda}_E(v):= \limsup_{n \to \infty} \frac{1}{n}\log \| E^n(v) \|,
\end{align*}
where $\|\cdot \|$ is any norm on $\bR^N$, and the result does not depend on this choice. 

The Lyapunov exponents can be computed in terms of the eigenvalues.

\begin{thm}[{\cite[Theorem 1.5.6]{CK}}]\label{t:Lyapunov}
For a non-zero vector $v = c_1v^1+ \dots +c_\ell v^\ell \in L(\lambda_1)\oplus \dots \oplus L(\lambda_\ell)$, the Lyapunov exponent of the orbit $\{E^n(v)\}_{n \geq 0}$ is expressed as
\begin{align*}
    \mathsf{\Lambda}_E(v) = \max_{j:~c_j \neq 0} \log\lambda_j.
\end{align*}
\end{thm}
As a special case, we have the following:

\begin{cor}\label{cor:Lyapunov_independent}
Suppose that vectors $v_1,\dots,v_N \in V$ form a basis of $V$. Then we have
\begin{align*}
    \max_{\indi=1,\dots,N} \mathsf{\Lambda}_E(v_\indi) = \log \rho(E).
\end{align*}
\end{cor}

\end{document}